\documentclass{svmult}
\usepackage{xargs,comment,amsopn,etex,bm} 

\usepackage{comment,amsmath,amssymb,changepage,pifont,graphicx,faktor,textpos,latexsym}
\usepackage{enumitem}

\usepackage[toc,page]{appendix}

\usepackage{mathptmx}       % selects Times Roman as basic font
\usepackage{helvet}         % selects Helvetica as sans-serif font
\usepackage{courier}        % selects Courier as typewriter font
\usepackage{type1cm}        % activate if the above 3 fonts are
\usepackage{makeidx}         % allows index generation
\usepackage{graphicx}        % standard LaTeX graphics tool
\usepackage{multicol}        % used for the two-column index
\usepackage[bottom]{footmisc}% places footnotes at page bottom

\usepackage[T1]{fontenc}
\usepackage[utf8]{inputenc}

\usepackage{mathrsfs}

\usepackage{amscd}
\usepackage[all]{xy}

\usepackage{tikz}
\usetikzlibrary{patterns,cd}

\usepackage{amsfonts}

\usepackage[colorlinks]{hyperref}

\DeclareMathOperator{\Cone}{Cone}
\DeclareMathOperator{\Diag}{Diag}
\DeclareMathOperator{\dist}{dist}
\DeclareMathOperator{\Gal}{Gal}
\DeclareMathOperator{\GL}{GL}
\DeclareMathOperator{\lcm}{lcm}
\DeclareMathOperator{\length}{length}
\DeclareMathOperator{\PRM}{PRM}
\DeclareMathOperator{\Proj}{Proj}
\DeclareMathOperator{\RM}{RM}
\DeclareMathOperator{\Sing}{Sing}
\DeclareMathOperator{\Vol}{Vol}
\DeclareMathOperator{\wdeg}{\w-deg}
\DeclareMathOperator{\WPRM}{WPRM}
\DeclareMathOperator{\wt}{wt}

\renewcommand{\AA}{\mathbb{A}}
\newcommand{\FF}{\mathbb{F}}
\newcommand{\GG}{\mathbb{G}}
\newcommand{\NN}{\mathbb{N}}
\newcommand{\PP}{\mathbb{P}}
\newcommand{\RR}{\mathbb{R}}
\newcommand{\TT}{\mathbb{T}}
\newcommand{\VV}{\mathbb{V}}
\newcommand{\ZZ}{\mathbb{Z}}
\newcommand{\Fq}{\mathbb{F}_q}

\newcommand{\aff}{\mathrm{aff}}
\newcommand{\Aff}{\mathbb{A}}
\newcommand{\bmu}{\bm{\mu}}
\newcommand{\card}[1]{\left\lvert#1\right\rvert}
\newcommand{\biggcard}[1]{\bigg\lvert#1\bigg\rvert}
\newcommand{\inv}{\mathrm{inv}}
\newcommand{\reg}{\mathcal{O}}
\newcommand{\sansS}{\mathsf{S}}
\newcommand{\set}[2]{\left\{#1\,:\,#2\right\}}
\newcommand{\w}{\mathsf{a}}

\makeindex
\smartqed

\begin{document}

\title{Hypersurfaces in weighted projective spaces over finite fields 
   with applications to coding theory}
\titlerunning{Hypersurfaces in weighted projective spaces}

\author{Yves Aubry,
       \hbox{Wouter Castryck},
       \hbox{Sudhir R.  Ghorpade},
       \hbox{Gilles Lachaud},
       \hbox{Michael E. O'Sullivan}, and
       \hbox{Samrith Ram}}
\authorrunning{Aubry, Castryck, Ghorpade, Lachaud, O'Sullivan, and Ram}

\institute{
     Yves Aubry 
\at  Institut de Math\'ematiques de Toulon (IMATH), 
     Universit\'e de Toulon, France and 
     Aix Marseille Univ., CNRS, Centrale Marseille, 
     I2M, Marseille, France.
\email{yves.aubry@univ-tln.fr}
\and Wouter Castryck 
\at  Laboratoire Painlev\'e, 
     Universit\'e de Lille-1, 
     Cit\'e Scientifique, 
     59\,655, Villeneuve d'Ascq, cedex, France and
     Departement Elektrotechniek imec-Cosic, 
     KU Leuven, 
     Kasteelpark Arenberg 10, 
     3001 Leuven, Belgium. 
\email{wouter.castryck@kuleuven.be}
\and Sudhir R.  Ghorpade 
\at  Department of Mathematics, 
     Indian Institute of Technology Bombay, Powai, 
     Mumbai 400076, India.
\email{srg@math.iitb.ac.in}
\and Gilles Lachaud 
\at  Aix Marseille Univ., CNRS, Centrale Marseille, 
     I2M, Marseille, France.
\email{gilles.lachaud@univmed.fr}
\and Michael E. O'Sullivan 
\at  Department of Mathematics and Statistics, 
     San Diego State University, 
     San Diego, CA 92182-7720, USA.
\email{mosullivan@mail.sdsu.edu}
\and Samrith Ram 
\at  Harish-Chandra Research Institute, 
     Chhatnag Road, Jhusi, 
     Allahabad 211019, India.
\email{samrithram@hri.res.in}
}

\maketitle

\abstract{
We consider the question of determining the maximum number of $\Fq$-rational 
points that can lie on a hypersurface of a given degree in a weighted projective
space over the finite field~$\Fq$, or in other words, the maximum number of 
zeros that a weighted homogeneous polynomial of a given degree can have in the 
corresponding weighted projective space over~$\Fq$. In the case of classical 
projective spaces, this question has been answered by J.-P.~Serre. In the case 
of weighted projective spaces, we give some conjectures and partial results. 
Applications to coding theory are included and an appendix providing a brief 
compendium of results about weighted projective spaces is also included. 
}

%\tableofcontents

\begin{comment}

\end{comment}

\section{Introduction}
\label{sec:intro}

Let $q$ be a prime power and let $\Fq$ denote the finite field with $q$
elements. Let $d \geq 0$ and $m \geq 1$ be integers. For any integer $r$, we 
define 
\[
p_r := \card{\PP^r(\Fq)} = q^r + q^{r-1} + \dots + 1 \quad \text{for }r\geq 0
       \quad \text{ and } \quad \mbox{$p_r:=0$ for } r < 0.
\] 
In a letter to M.~Tsfasman  in 1989, J.-P.~Serre~\cite{letterserre} proved that
for any nonzero homogeneous degree $d$ polynomial 
$F \in \Fq[X_0, X_1, \ldots, X_m]$, the hypersurface $V(F)$ consisting of 
$\Fq$-rational zeros of $F$ in the projective $m$-space $\PP^m$ satisfies
\begin{equation} 
\label{serresinequality}
\card{V(F)}\le d q^{m-1} + p_{m-2}.
\end{equation}
Note that if $d \ge q+1$, then $d q^{m-1} + p_{m-2} \ge p_m = \card{\PP^m(\Fq)}$, 
and thus the above bound is trivial in this case; moreover, the polynomial 
$X_0^{d-q-1}(X_0^qX_1 - X_0X_1^q)$ is evidently homogeneous of degree 
$d\ge q+1$ and has $p_m$ zeros in $\PP^m(\Fq)$. On the other hand, in the
nontrivial case when $d \le q+1$, the bound \eqref{serresinequality} is met by
\begin{equation} 
\label{serresinequalitymet}
F = \prod_{i=1}^d (\alpha_i X_0 - \beta_i X_1),
\end{equation}
whenever $(\alpha_1 : \beta_1), (\alpha_2 : \beta_2), \ldots, (\alpha_d : \beta_d)$
are distinct elements of $\PP^1(\Fq)$. It follows that if we let $e_q(d,m)$
denote the maximum possible number of $\Fq$-rational zeros in $\PP^m$ that a
nonzero homogeneous  polynomial of degree $d$ in $\Fq[X_0, X_1, \ldots, X_m]$  
can admit, then  
\begin{equation} 
\label{serreequality}
e_q(d,m) = \min \{ p_m, \; d q^{m-1} + p_{m-2} \}. 
\end{equation}
Alternative proofs of \eqref{serresinequality}, and hence \eqref{serreequality},
can be found in \cite{So} and~\cite{DG1}, whereas some extensions and 
generalizations are given in \cite{C} and~\cite{DG2}.  Serre's result has also
been applied to determine the minimum distance of the projective Reed--Muller 
codes, which were introduced by Lachaud in~\cite{lachaud}, and further studied 
in \cite{L} and~\cite{So}. 

In this paper we discuss how the bound~\eqref{serresinequality} can possibly be
generalized to weighted projective spaces, along with a number of partial 
results and some implications for coding theory.  Let us recall that given any 
positive integers $a_0, a_1, \ldots, a_m$, the corresponding weighted projective
space is defined by
\[ 
\PP(a_0, a_1, \ldots, a_m) := \left( \overline{\FF}_q^{m+1} \setminus \{ (0, 0, \ldots, 0) \} \right) / \sim
\]
where $ \overline{\FF}_q$ denotes an algebraic closure of $\Fq$ and the
equivalence relation $\sim$ is such that
\[ 
(x_0, x_1, \ldots, x_m) \sim (\lambda^{a_0} x_0 , \lambda^{a_1} x_1 , \ldots, \lambda^{a_m} x_m )
  \quad \text{ for every } \lambda \in \overline{\FF}_q^\ast. 
\]
The corresponding equivalence class is denoted by $(x_0 : x_1 : \cdots : x_m)$
and is called a weighted projective point. We say that the point is 
$\Fq$-rational if $(x_0 : x_1: \cdots : x_m) = (x_0^q : x_1^q : \cdots : x_m^q)$.
It can be shown using Hilbert's theorem 90 that every $\Fq$-rational point has
at least one representative in $\Fq^{m+1} \setminus \{ (0, 0, \ldots, 0) \}$. 
In fact, a finer analysis shows that it has exactly $q-1$ such representatives; 
see \cite[\S3]{perret1}. In particular, the total number of $\Fq$-rational
points equals $p_m$, i.e.~it is the same as in the non-weighted case. The 
weighted projective spaces are fascinating objects. On the one hand, they are
analogous to classical projective spaces, but they are often difficult to deal
with, partly since they can admit singularities. For the convenience of the 
reader, and possible future use, we include at the end of this paper a fairly 
self-contained appendix that provides a glossary of various notions and results 
concerning weighted projective spaces. 

Now let $\sansS = \Fq[X_0, X_1, \ldots, X_m]$ and consider a nonzero 
polynomial $F \in \sansS$ which is homogeneous of degree $d$ provided that
we measure $X_i$ with weight $a_i$ for each $i = 0, 1, \ldots, m$, so that
\[ 
F(\lambda^{a_0} X_0, \lambda^{a_1} X_1, \ldots, \lambda^{a_m} X_m) 
   = \lambda^d F(X_0, X_1, \ldots, X_m) 
     \quad \text{for all }\lambda \in \overline{\FF}_q^\ast. 
\]
Thus it is meaningful to consider the weighted projective hypersurface $V(F)$ 
of $\Fq$-rational points of $\PP(a_0, a_1, \ldots, a_m)$ at which $F$ 
vanishes. Our object of study is the quantity
\[ 
e_q(d; a_0,a_1, \ldots, a_m) := \max_{F \in \sansS_d \setminus \{0\} } \card{V(F)}, 
\]
where $\sansS_d$ denotes the space of weighted homogeneous polynomials in 
$\sansS$ of degree~$d$. One caveat is that $\sansS_d$ might be trivial
for certain values of $d$ (namely those values that are not contained in the
semigroup 
$a_0 \ZZ_{\geq 0} + a_1 \ZZ_{\geq 0} + \ldots + a_m \ZZ_{\geq 0}$),
in which case we say that $e_q(d; a_0,a_1, \ldots, a_m)$ is \emph{not defined}. 
Also note that $e_q(d; a_0, a_1, \ldots, a_m)$ is not necessarily increasing as 
a function in $d$: for instance $e_q(7;3,4) = 2$ while $e_q(8;3,4) = 1$ since 
the only monomials of (weighted) degree 7 and 8 are constant multiples of 
$X_0X_1$ and $X_1^2$ respectively.

Seeking inspiration in the example~\eqref{serresinequalitymet} that meets 
Serre's bound, it is natural to consider polynomials of the form
\begin{equation} 
\label{serresinequalitygen}
F = \prod_{i=1}^{d/a_{rs}} (\alpha_i X_r^{a_{rs}/a_r} - \beta_i X_s^{a_{rs}/a_s}),
\end{equation}  
where $r,s \in \{0, 1, \ldots, m \}$ are distinct indices, $a_{rs}$ is the least
common multiple of $a_r$ and $a_s$, $d$ is a multiple of $a_{rs}$ satisfying 
$d\le a_{rs}(q+1)$, and the $(\alpha_i : \beta_i)$'s are distinct elements of
$\PP^1(\Fq)$. In Section~\ref{section_preliminary} we will prove that
$\card{V(F)} = \left( {d}/{a_{rs}}\right) q^{m-1} + p_{m-2}$, leading to the 
following lower bound:

\begin{lemma} 
\label{serrelowerbound}
Let $a = \min \set{\lcm(a_r,a_s)}{0 \leq r < s \leq m}$ and 
assume that $a \mid d$. Then
\[ 
e_q(d;a_0,a_1, \ldots, a_m) \geq \min \left\{ p_m,  \; \frac{d}{a} q^{m-1} + p_{m-2} \right\}. 
\]
\end{lemma}

\begin{example}
Let us prove that equality holds in the lemma for $\PP(a_0,a_1)$.  Writing 
$a = \lcm(a_0,a_1)$, we  want to prove that $e_q(d;a_0,a_1) = \min \{p_1,d/a\}$.
Let $F \in  \sansS_d\setminus \{0\}$ and   note that
\[
F(X_0,X_1)/X_1^{d/a_1} 
\] 
can be viewed as a univariate polynomial in $T = X_0^{a/a_0} / X_1^{a/a_1}$. 
Indeed, if a monomial $ X_0^{\beta_0} X_1^{\beta_1} $ is weighted homogeneous 
of degree $d$, so that  $\beta_0 a_0 + \beta_1 a_1 = d$, then an easy 
calculation shows that 
\[ 
\frac{X_0^{\beta_0} X_1^{\beta_1} }{X_1^{d/a_1}} 
   = \left( \frac{X_0^{a/a_0}}{X_1^{a/a_1}} \right)^{\beta_0 a_0 / a }. 
\]
Let $d=ak$ and $b_i = a/a_i$ for $i=0,1$.  Now factor $F(X_0,X_1)/X_1^{b_1k}$ 
and remultiply with $X_1^{b_1k}$ to obtain 
\[
F(X_0, X_1) = c \cdot X_1^{b_1 \ell}\cdot \prod_{i=1}^{ k-\ell} \left( X_0^{b_0} - t_i X_1^{b_1} \right)  
\]
for some $\ell \leq k$, some $t_i \in \overline{\FF}_q$ and some leading 
coefficient $c \in \Fq^\ast$.  Each factor for which $t_i \in \Fq$ has a
unique $\Fq$-rational zero in $\PP(a_0,a_1)$. Indeed, to see this it suffices
to show that such a factor has exactly $q - 1$ solutions 
$(X_0,X_1) \in \Fq^2 \setminus \{(0,0) \}$, which easily follows from the 
coprimality of $b_1, b_2$; see also Lemma~\ref{bivariatecountlemma} below. On
the other hand, a factor for which $t_i \notin \Fq$ clearly cannot have any 
$\Fq$-rational zeroes. This shows that $e_q(d;a_0,a_1)=  k= d/a$ for $d\leq q+1$. 
\end{example}

\noindent 
In Section~\ref{section_preliminary} we will generalize the class of 
polynomials~\eqref{serresinequalitygen} to a larger family which shows that the
inequality may be strict if $m > 1$. We prudently conjecture that the actual 
value of $e_q(d;a_0,a_1,\ldots,a_m)$ is always attained by one of these 
generalizations (as soon as it is defined), but elaborating this into a concrete
statement amounts to tedious additive number theory and is omitted.

One assumption that simplifies the combinatorics is 
$\lcm(a_0,a_1, \ldots, a_m) \mid d$; in what follows we will usually suppose 
that this is the case. Another hypothesis which turns out to simplify things 
significantly is that one of the weights (say $a_0$) equals $1$. Under these
assumptions, we conjecture:

\begin{conjecture} 
\label{serreimplicitconjecture}
\emph{
If $a_0 = 1$ and $\lcm(a_1,a_2, \ldots, a_m) \mid d$, then the bound from 
Lemma~\textup{\ref{serrelowerbound}} is sharp.  In other words, if we order the weights
such that $a_1 \leq a_2 \leq \ldots \leq a_m$, then
\[ 
e_q(d;1,a_1,a_2,\ldots,a_m) = \min \left\{ p_m,  \frac{d}{a_1} q^{m-1} + p_{m-2} \right\}. 
\]
}
\end{conjecture}

\noindent This immediately specializes to Serre's bound for $a_1 = \ldots = a_m = 1$.
The right-hand side equals $\frac{d}{a_1} q^{m-1} + p_{m-2}$ if and only if
$d \leq a_1(q + 1)$, which will be assumed in practice because the other case 
is again easy to handle. 

In the statement of Conjecture~\ref{serreimplicitconjecture} it can be assumed 
without loss of generality that $\gcd(a_1, a_2, \ldots, a_m) = 1$. This follows 
from Delorme weight reduction~\cite{Delorme}, which states that for any index 
$i$ and any positive integer $b$ coprime to $a_i$, 
\[
\PP(a_0b, \dots,a_{i-1}b, a_i, a_{i+1}b, \dots, a_mb) \cong \PP(a_0, a_1, \dots, a_m),
\]
the underlying observation being that an 
$(a_0b, \dots,a_{i-1}b, a_i, a_{i+1}b, \dots, a_mb)$-weighted homogeneous
polynomial of degree $d = kb$ (with $k$ some integer) can be easily transformed 
into an $(a_0b, \dots,a_{i-1}b, a_ib, a_{i+1}b, \dots, a_mb)$-weighted
homogeneous polynomial of the same degree, by replacing each occurrence of 
$X_i^b$ by~$X_i$. A rescaling of the weights then allows us to view this as an 
$(a_0,a_1, \ldots, a_m)$-weighted homogeneous degree $k$ polynomial. See the 
treatments in \cite[\S3.3]{Hos}, \cite[\S3.6]{R}, \cite[\S1]{dolgachev} for
more details.  For our needs, the relevant observation is that there is a 
one-to-one correspondence between the respective $\Fq$-rational zeroes given
by
\[ 
(\alpha_0 : \ldots : \alpha_{i-1} : \alpha_i : \alpha_{i+1} : \ldots : \alpha_m ) 
\ \mapsto 
\ (\alpha_0 : \ldots : \alpha_{i-1} : \alpha_i^b : \alpha_{i+1} : \ldots : \alpha_m ).
\]
In particular the Delorme isomorphism respects 
Conjecture~\ref{serreimplicitconjecture} in the sense that  
$ e_q(db;1,a_1b,a_2b,\ldots,a_mb)$ and $e_q(d;1,a_1,a_2,\ldots,a_m)$ have the
same value.

For $m=1$, the validity of Conjecture~\ref{serreimplicitconjecture} follows
from the example discussed above; we note that alternatively this example could
have been settled by reducing to the case of $\PP^1(1,1)$ using Delorme weight
reduction (preceded by a rescaling of the weights if needed to ensure that 
$\gcd(a_0,a_1)=1$). In Section~\ref{section_mainsection} we give further 
evidence in favour of Conjecture~\ref{serreimplicitconjecture}:

\begin{theorem}
\label{serreforplanes}
Conjecture~\textup{\ref{serreimplicitconjecture}} is true if $m \leq 2$.
\end{theorem}
The proof for $m=2$ is done by mimicking Serre's original method. In order to
do so, our main task is to come up with a convenient notion of `lines' inside 
the weighted projective plane, which is not obvious a priori. The handy 
property of $\PP(1,a_1,a_2)$ is that it naturally arises as a completion of the 
affine plane $\AA^2$, which leads us to consider completed affine lines; as we
will see, these indeed allow for a working version of Serre's proof. Even
though $\PP(1,a_1,a_2)$ is a very particular case, we hope that our approach has
the ingredients needed to establish Conjecture~\ref{serreimplicitconjecture} in
full generality. 

Finally, in Section~\ref{section_WPRM}, we introduce the natural weighted 
analogue of projective Reed--Muller codes, reinterpret 
Conjecture~\ref{serreimplicitconjecture} in terms of the minimal distance, and 
examine some further first properties. These codes do not seem to have seen 
previous study, even though a different notion bearing the name 
`weighted projective Reed--Muller codes' was introduced and analyzed by 
S{\o}rensen~\cite{sorensen}. As noted earlier, an appendix giving a formal 
introduction to weighted projective spaces and many of its geometric aspects is
provided at the end. 

\section{Polynomials with many zeros} 
\label{section_preliminary}

In this section we generalize the class of polynomials considered 
in~\eqref{serresinequalitygen}. As before, let $\sansS$ denote the polynomial
ring $\Fq[X_0, X_1, \ldots, X_m]$. Fix a grading on $\sansS$ with respect 
to weights $\w = (a_0, a_1, \dots , a_m)$ so that $\deg X_i = a_i \ge 1$ 
($0\le i \le m$), and for a monomial $M= X_0^{i_0}X_1^{i_1} \cdots X_m^{i_m}$,
the (weighted) degree of $M$ is $\deg M = i_0a_0+i_1a_1+ \dots + i_ma_m$.  We 
now define a useful notion about pairs of monomials in~$\sansS$. 

\begin{definition}
Let $M_0,M_1 \in \sansS$ be monomials different from $1$. If
\begin{itemize} 
\item $\deg M_0 = \deg M_1$,
\item $\gcd(M_0,M_1) = 1$, i.e.\ no variables appear in both $M_0$ and $M_1$,
\item $\gcd(\text{exponents appearing in the monomial $M_0 M_1$}) = 1$,
\end{itemize}
then we call $( M_0,M_1)$ a \emph{primitive pair}. Denoting by $s_i$ ($i=0,1$) 
the number of distinct variables appearing in $M_i$, we call $(s_0,s_1)$ the 
corresponding \emph{signature}.
\end{definition} 

\begin{example}
For $\PP(2,3,5)$, the pairs $( X_0X_1, X_2 )$, $( X_0^3, X_1^2 )$ are primitive 
of degrees $5$, $6$ and signatures $(2,1)$, $(1,1)$, respectively.
\end{example}

Our generalized class consists of weighted homogeneous polynomials of the form
\begin{equation} 
\label{serrenewexample}
F_{\ell,s_0,s_1,\sigma_0,\sigma_1} = \mu_0 \mu_1 \prod_{i=1}^\ell (M_0 - t_i M_1)  
\end{equation}
where $1 \leq s_0 \geq \sigma_0 \geq 0$, $1 \leq s_1 \geq \sigma_1 \geq 0$
are integers and
\begin{itemize}
\item $(M_0, M_1)$ is a primitive pair of signature $(s_0,s_1)$,
\item $t_1,  \ldots, t_\ell$ are distinct elements of $\Fq^\ast$ 
      (in particular $0 \leq \ell \leq q-1$),
\item the (possibly trivial) monomial $\mu_i$ ($i=0,1$) is only divisible by 
      variables that also appear in $M_i$; more precisely it is divisible by 
      $\sigma_i \leq s_i$ such variables.
\end{itemize}
It is allowed that $\ell = 0$, but in that case we assume that $\sigma_0 = s_0$
and $\sigma_1 = s_1$. In this case $F$ is just a monomial in at least two
variables. Strictly speaking, since we assumed that $s_0\ge1$ and $s_1\ge 1$, 
monomials in one variable (or $F=1$) are not covered by the construction, but 
in order to have a chance of meeting $e_q(d;a_0,a_1,\ldots, a_m)$ for every 
value of $d$ one should include them; since this is speculative anyway, we 
omit a further discussion of such pathologies.

The construction indeed concerns a generalization of~\eqref{serresinequalitygen}:
modulo scaling, the polynomial 
\[
\prod_{i=1}^{d/a_{rs}} (\alpha_i X_r^{a_{rs}/a_r} - \beta_i X_s^{a_{rs}/a_s})
\]
is of the form $F_{d/a_{rs} - \sigma_0 -\sigma_1,1,1,\sigma_0,\sigma_1}$ with 
$\sigma_0, \sigma_1 \in \{0,1\}$, depending on whether $(1:0)$ or $(0:1)$ are
among the points $(\alpha_i : \beta_i)$. Here the underlying primitive pair is 
$(X_r^{a_{rs}/a_r}, \;  X_s^{a_{rs}/a_s})$. 

Of course the polynomial $F_{\ell,s_0,s_1,\sigma_0,\sigma_1}$ is not uniquely 
determined by the integers $\ell,s_0,s_1,\sigma_0,\sigma_1$, but these are the
parameters accounting for the number of $\Fq$-rational points at which it 
vanishes: 

\begin{lemma} 
\label{pointcountinglemma}
$ \card{V(F_{\ell,s_0,s_1,\sigma_0,\sigma_1})} = \lambda q^{m + 1 -s_0-s_1} +  p_{m-s_0 - s_1}$
where
\begin{align*}
 \lambda  & =  \ell \cdot (q-1)^{s_0+s_1 - 2} \\
 &\qquad + \, \left[ (q^{s_0} - (q-1)^{s_0}) (q^{s_1} - (q-1)^{s_1}) - 1 \right] / (q-1) \\
 &\qquad + \,  (q-1)^{s_1 - 1} q^{s_0 - \sigma_0} (q^{\sigma_0} - (q-1)^{\sigma_0}) \\
 & \qquad + \, (q-1)^{s_0 - 1} q^{s_1 - \sigma_1} (q^{\sigma_1} - (q-1)^{\sigma_1}).
\end{align*}
\end{lemma}

In order to prove this, let us denote the variables appearing in $M_0$ and 
$M_1$ by $Y_1, Y_2, \ldots, Y_{s_0}$ and $Z_1, Z_2, \ldots, Z_{s_1}$,
respectively. These are distinct because of the primitivity of the pair 
$(M_0,M_1)$. The points at which all these variables vanish have the structure 
of a weighted projective space of dimension $m-s_0-s_1$. Since there are 
$p_{m-s_0-s_1}$ such points which are $\Fq$-rational, our task easily reduces
to the case where $s_0 + s_1 = m + 1$, meaning that each of the variables 
$X_0, X_1, \ldots, X_m$ appears among the $Y_i$ or $Z_i$. In the latter case we
need to show that $\card{V(F_{\ell,s_0,s_1,\sigma_0,\sigma_1})} =  \lambda$.  We 
claim that, respectively, the summands in the statement of
Lemma~\ref{pointcountinglemma} correspond to

\begin{enumerate}[label=(\roman{enumi}), leftmargin=*]
\item the zeros all of whose coordinates are nonzero,
\item the zeros for which at least one of the $Y_i$'s is zero 
      and at least one of the $Z_i$'s is zero,
\item the zeros for which at least one of the $Y_i$'s is zero, 
      but none of the $Z_i$'s is,
\item the zeros for which at least one of the $Z_i$'s is zero, 
      but none of the $Y_i$'s is.
\end{enumerate}

As for (i), this immediately follows from the lemma below, along with the
primitivity of $(M_0,M_1)$ and the fact that every $\Fq$-rational weighted 
projective point has exactly $q-1$ rational representatives by~\cite[\S3]{perret1}.

\begin{lemma} 
\label{bivariatecountlemma}
Let $a_1, a_2, \ldots, a_{s_0},b_1,b_2, \ldots, b_{s_1}$ be mutually coprime 
integers and let $\alpha, \beta \in \Fq^\ast$. Then the number of solutions in
the torus $\TT_q^{s_0 + s_1}(\Fq) := (\Fq^\ast)^{s_0 + s_1}$ of the equation
\[ 
\alpha x_1^{a_1}x_2^{a_2} \cdots x_{s_0}^{a_{s_0}} - 
 \beta y_1^{b_1}y_2^{b_2} \cdots y_{s_1}^{b_{s_1}} = 0
\]
is given by $(q-1)^{s_0 + s_1 - 1}$.
\end{lemma}

\begin{proof}
Since $a_0,a_1, \ldots, a_{s_0}, -b_0, -b_1, \ldots, -b_{s_1}$ are coprime, 
these integers can be viewed as the entries in the first row of a matrix 
$M \in \GL_{s_0+s_1}(\ZZ)$; see~\cite{konrad}.  Rewrite the equation as
\[ 
x_1^{a_1}x_2^{a_2} \cdots x_{s_0}^{a_{s_0}} y_1^{-b_1}y_2^{-b_2} \cdots y_{s_1}^{-b_{s_1}}
    = \alpha^{-1} \beta.
\]
Using $M$ it is easy to find a monomial transformation (= an invertible 
substitution of the variables by Laurent monomials) that takes this equation to
\[ 
x_1 = \alpha^{-1} \beta. 
\] 
This transformation determines a bijection between the respective sets of 
solutions inside $\TT^{s_0 + s_1}(\Fq)$, from which the lemma follows. \qed
\end{proof}

As for (ii), note that if a point 
$(y_1 : y_2 : \ldots : y_{s_0} : z_1 : z_2 : \ldots : z_{s_1})$ satisfies 
$y_i = 0$ and  $z_j = 0$ for at least one pair $y_i,z_j$ then it automatically
concerns a zero of $F_{\ell,s_0,s_1,\sigma_0,\sigma_1}$. There are
\[
(q^{s_0} - (q-1)^{s_0}) (q^{s_1} - (q-1)^{s_1}) - 1
\]
such points in $\Fq^{s_0 + s_1} \setminus \{(0,0, \ldots, 0)\}$, and so we find
the desired contribution, again by using that every $\Fq$-rational point has 
$q-1$ representatives. 

Concerning (iii): these are exactly the zeros of $\mu_0$ that were not counted
elsewhere. Once more we adopt the strategy of first counting the number of 
$\Fq$-rational representatives, after which we divide by $q-1$. At least one 
of the $\sigma_0$ variables appearing in $\mu_0$ should be set to zero, 
accounting for the factor $q^{\sigma_0} - (q-1)^{\sigma_0}$, while the other 
$Y_i$'s can be chosen freely and the $Z_i$'s must be chosen nonzero, accounting 
for the factors $q^{s_0 - \sigma_0}$ and $(q-1)^{s_1}$, respectively.

The case (iv) follows by symmetry. This completes the proof of 
Lemma~\ref{pointcountinglemma}.

\begin{example}
Consider $\PP(2,3,5)$, let $d = 30$, and assume $q \geq 5$. Let
\[
F_{4,2,1,2,1} = X_0X_1X_2 \prod_{i=1}^4 (X_0X_1 - t_i X_2). 
\]
According to Lemma~\ref{pointcountinglemma}, the  number of $\Fq$-rational zeros
of $F_{4,2,1,2,1}$  is $7q - 4$. We believe that this equals $e_q(30;2,3,5)$, 
although we currently cannot offer a proof. But at least this shows that the 
lower bound from Lemma~\ref{serrelowerbound}, which relied on the polynomial
\[
F_{3,1,1,1,1} = X_0^3X_1^2 \prod_{i=1}^3 (X_0^3 - t_i X_1^2), 
\]
can be strict: indeed, $F_{3,1,1,1,1}$ has only $5q + 1$ zeros. On the other 
hand, for $q=4$, this last polynomial trivially meets $e_q(30;2,3,5)$ because 
it is `space-filling', i.e.,  its set of $\Fq$-rational zeros equals all of 
$\PP(2,3,5)(\Fq)$.
\end{example}

\section{Hypersurfaces in Weighted Projective Planes $\PP(1,a_1,a_2)$} 
\label{section_mainsection}

In this section we prove Theorem~\ref{serreforplanes}, i.e.\@ we prove 
Conjecture~\ref{serreimplicitconjecture} for weighted projective planes 
$\PP(1,a_1,a_2)$. Note that by Serre's result for classical projective spaces 
and by Delorme's isomorphism we may assume without loss of generality that 
$a_1 < a_2$ and that these weights are coprime, so $\lcm(a_1,a_2) = a_1a_2$.  
Let $ F \in \Fq[X_0,X_1,X_2] $ be a nonzero polynomial which is weighted 
homogeneous of degree $d$ with $a_1a_2 \mid d$. Assuming that $d \leq a_1(q+1)$,
our task is to prove
\begin{equation} 
\label{serreboundplane}
\card{V(F)} \leq \frac{d}{a_1} q + 1.
\end{equation}
This we will do by mimicking Serre's original proof, for which we need a 
convenient notion of `lines' in the weighted projective plane. Note that if we 
define lines merely as subsets that are cut out by a weighted homogeneous
polynomial of degree $1$, in general the resulting notion is too poor to be of 
any use (we would usually only find $X_0 = 0$).

An easy but crucial feature of having $a_0 = 1$ is that every point 
$(x_0 : x_1 : x_2)$ for which $x_0 \neq 0$ has a unique representative of the
form $(1 : x : y)$. Moreover, the point is $\Fq$-rational if and only if 
$x,y \in \Fq$. Thus the embedding
\[ 
\AA^2 \hookrightarrow \PP(1,a_1,a_2) : (x,y) \mapsto (1 : x : y) 
\]
identifies $\AA^2$ with the chart $X_0 \neq 0$, in an equivariant way 
(i.e.\@ the identification continues to hold if one restricts to 
$\Fq$-rational points).  We call $H_\infty : X_0 = 0$ the `line at infinity'.
Note that it naturally carries the structure of the weighted projective line 
$\PP(a_1,a_2)$.

\begin{remark}
We can think of $\PP(1,a_1,a_2)$ as the affine plane to which a line at infinity
has been glued, albeit in a non-standard way. This can be made precise 
geometrically (see, for example,\ Dolgachev~\cite{dolgachev}) and it turns out 
(see, for example, Section 2 of the appendix) that, in general, the coordinate
points at infinity are singular (we will not use this).
\end{remark}

\begin{remark}
Writing $V(F)^{\aff}$ for the set of affine $\Fq$-rational zeroes, it is not
too hard to show that $\card{V(F)^{\aff}} \leq (d/a_1)q$, for instance using 
Ore's inequality; 
%see Section~4 of the appendix.
see Section~\ref{app:ore} of the appendix.
\end{remark}

The affine zeros of $F$ are precisely the zeros of the dehomogenized polynomial
\[ 
F(1,x,y) \in \Fq[x,y]. 
\]
Conversely, given a polynomial in $x$ and $y$, there is a natural way of 
homogenizing it, by substituting $x \leftarrow X_1, y \leftarrow X_2$ and 
adding to each term as many factors $X_0$ as minimally needed. We define a 
`line' in $\PP(1,a_1,a_2)$ to be either a homogenized linear bivariate equation,
or the line at infinity:

\begin{definition} 
\label{defline}
An \emph{$\Fq$-rational line} in $\PP(1,a_1,a_2)$ is a subset defined by an 
equation of one of the following types.
\begin{itemize}
\item Type $0$: The line $X_0 = 0$, which we shall denote $H_\infty$ 
      (the \emph{line at infinity}). Points on this line may be called the 
      \emph{points at infinity}. 
\item Type $1$: Lines of the form $\alpha X_0^{a_1} + X_1 = 0$ with 
      $\alpha \in \Fq$ (\emph{vertical} lines).
\item Type $2$: Lines of the form 
      $\alpha X_0^{a_2} + \beta X_1 X_0^{a_2 - a_1} + X_2 = 0$ with
     $\alpha, \beta \in \Fq$ (\emph{non-vertical} lines).
\end{itemize}
\end{definition}

\begin{remark}
\label{rem:changeofvariable}
Note that using an $\Fq$-rational change of variables that respects the grading,
any $\Fq$-rational line of type $i$ can be transformed into $X_i = 0$.  For 
instance, for the vertical line $\alpha X_0^{a_1} + X_1 = 0$ this amounts to 
substituting $X_1 \leftarrow X_1 - \alpha X_0^{a_1}$.
\end{remark}

\begin{lemma} 
\label{linepropertieslemma}
Any $\Fq$-rational line in $\PP(1,a_1,a_2)$ contains exactly $q+1$ rational 
points, and any pair of $\Fq$-rational lines in $\PP(1,a_1,a_2)$ has at least 
one rational point in common.
\end{lemma} 

\begin{proof}
Being a copy of $\PP(a_1,a_2)$, it is clear that the line at infinity in
$\PP(1,a_1,a_2)$ contains $q+1$ rational points, while all other $\Fq$-rational
lines contain $q$ affine points along with a unique point at infinity.  Clearly
type~1 and type~2 lines meet the line $X_0=0$ and a type~1 line meets a type~2 
line in the affine plane.  Type~1 lines all meet at $(0:0:1)$ and type~2 lines
all meet at $(0:1:0)$.  This establishes the lemma.\qed
\end{proof}

The points at infinity $(0:0:1)$ and $(0:1:0)$ on the coordinate axes will be
denoted by $P_\infty$ and $P_\infty'$, respectively.

\begin{remark} 
\label{remarkprojplane}
Figure~\ref{linesinWPS} illustrates the intersection behaviour of lines in 
$\PP(1,a_1,a_2)$; the point $P_\infty'$ acts as a vortex attracting all lines 
of type $2$.
\begin{figure}[h!] 
\centering
			\begin{tikzpicture}[scale=0.6]
			  \draw[-] (0,0) -- (10,0);
			  \draw[-] (0,0) -- (0,6);
              \draw [fill=black] (0,0) circle (0.1);
              \draw [cyan,fill=cyan] (1.5,1.5) circle (0.1);
              \draw [orange,fill=orange] (10,0) circle (0.1);
              \draw [orange,fill=orange] (0,6) circle (0.1); 
              \draw [cyan,thick] plot [smooth,tension=0.7] coordinates{(0,1.3) (1.5,1.5) (5.9,1) (10,0)}; 
              \draw [cyan,thick] plot [smooth,tension=0.7] coordinates{(0,0.5) (1.5,1.5) (5,1.9)  (10,0)}; 
              \draw [orange,thick] plot [smooth,tension=0.7] coordinates{(0,6) (4.5,2.8)  (10,0)};  
              \draw [purple,thick] (0,6) -- (4,0);            
              \node at (-0.5,-0.5) {$(1:0:0)$};
              \node at (10.5,-0.5) {$(0:1:0) = P'_\infty$}; 
              \node at (-1,6.4) {$P_\infty = (0:0:1)$};              
              \node at (5,3) {$H_\infty$};
\draw [->] (5.6,-0.6) arc[x radius=1.5cm, y radius =1.8cm, start angle=0, end angle=40];
			  \node at (5.6,-1) {affine plane};
			  \node at (6,6) {\begin{tabular}{l} \textcolor{orange}{type 0 ($H_\infty$)} \\ \textcolor{purple}{type 1 (vertical line)} \\
			  \textcolor{cyan}{type 2 (non-vertical line)} \end{tabular}};
			\end{tikzpicture}
    \caption{Lines in $\PP(1,a_1,a_2)$.} \label{linesinWPS}
\end{figure}
\end{remark}

We are now ready to prove the upper bound for $\card{V(F)}$ stated 
in~\eqref{serreboundplane}. Let $H_1, H_2, \ldots, H_t \in \Fq[X_0,X_1,X_2]$ be
the distinct `linear' factors of $F$, i.e.\@ the divisors of $F$ having one of
the three forms mentioned in Definition~\ref{defline}. Note that 
\[ 
d \geq \deg H_1H_2 \cdots H_t \geq 1 + (t-1)a_1, 
\]
which leads to  $t \leq d/a_1$ since $a_1 \mid d$. For each $i = 1, 2, \ldots, t$ 
we define $L_i=V(H_i)$, and we similarly write $L_\infty = V(X_0)$ for the set
of $\Fq$-rational points on $H_\infty$. Let
\[
L=\bigcup_{i=1}^{t}L_i.
\]
As a first step in the proof, we show that $\card{L} \leq tq+1$ by induction on $t$.
The case $t=0$ is trivial and the case $t=1$ follows from
Lemma~\ref{linepropertieslemma}. In the general case we have 
\begin{align*}
\card{L} &=    \biggcard{\bigcup_{i=1}^{t}L_i  }                                                         \\
         &=    \biggcard{\bigcup_{i=1}^{t-1}L_i} + \card{L_t} - \biggcard{\bigcup_{i=1}^{t-1}L_i\cap L_t}\\
         &\leq (t-1)q + 1 + q + 1 - 1                                                                    \\
         &=     tq + 1,
\end{align*}
where the second step again uses Lemma~\ref{linepropertieslemma}.

To proceed, we distinguish between three cases.\\

\noindent 
\textcolor{blue}{\textbf{Case 1}}:
Suppose that $V(F)\setminus L \subseteq L_\infty\setminus \{P_\infty\}.$
\begin{enumerate}
\item If $L_i=L_\infty$ for some $i$, then we have 
      \[ 
      \card{V(F)} = \card{L} \leq (d/a_1) q+1
      \]
      by the previous observation. 
\item Suppose $L_i\neq L_\infty$ for all $i$. Then:
     \begin{itemize}
     \item either $t=d/a_1$, which is possible only if all $H_i$'s are 
           vertical and $V(F)=L$, so again the bound follows 
           (note that this case covers our example~\eqref{serresinequalitygen} 
           proving sharpness),
     \item or $t<d/a_1$, in which case the following estimate applies:
           \begin{align*}
           \card{V(F)} &\leq \card{L} + \card{L_\infty \setminus \{P_\infty\}}  \\
                       &=    \card{L} + q                                      \\  
                       &\leq tq + 1 + q                                         \\
                       &\leq (d/a_1 - 1)q + 1 + q                               \\
                       &=    (d/a_1)q + 1.
           \end{align*}
     \end{itemize}
\end{enumerate}
This concludes the proof in Case 1.\\

\noindent 
\textcolor{blue}{\textbf{Case 2}}: 
There exists a point $P\in \AA^2$ that lies in $V(F)\setminus L$.  Let $X$ 
denote the set of pairs $(P',H)$ of $\Fq$-rational points and $\Fq$-rational 
lines such that $P,P' \in V(F) \cap H$ and $P \neq P'$. We are going to estimate
the cardinality of $X$ in two ways. On the one hand
\begin{align*}
\card{X} &=    \sum_{P'\in V(F)\setminus \{P\}} \card{\set{L}{L \mbox{ is a line with }P,P'\in L}}                               \\
         &\geq \sum_{P'\in V(F)^{\aff}\setminus \{P\}}1 \ \ = \ \ \card{V(F)^{\aff}\setminus \{P\}}, 
\end{align*}
where as before 
$V(F)^{\aff} = V(F) \cap \AA^2 = V(F) \setminus L_\infty$.
On the other hand, we have
\begin{align*}
\card{X} &=\sum_{\substack{H \ni P \\H \mbox{ \scriptsize{type} } 1 }}(\card{V(F)\cap H}- 1)+\sum_{\substack{H \ni P \\H \mbox{ \scriptsize{type} } 2 }}(\card{V(F)\cap H}- 1) \\   
         & \hspace{2.3cm} \downarrow \text{\footnotesize{$X_1 = 0 \rightsquigarrow \PP(1,a_2)$}} \hspace{1.15cm} \downarrow \text{\footnotesize{$X_2 = 0 \rightsquigarrow \PP(1,a_1)$}}\\
         &\leq \hspace{0.9cm} 1 \cdot \left(\frac{d}{a_2}-1 \right) \hspace{0.8cm} + \hspace{0.8cm} q\left(\frac{d}{a_1}-1\right).
\end{align*}
The first vertical arrow above indicates that in order to estimate 
$\card{V(F) \cap H}$ for a line $H$ of type $1$, we can assume that $H$ is 
defined by $X_1 = 0$, by using a change of variables if needed by the remark 
after Definition~\ref{defline}. But then our task is to estimate the number of 
$\Fq$-rational zeros of $F(X_0,0,X_2)$ in the weighted projective line 
$\PP(1,a_2)$, which is bounded by $d/a_2$ as observed in the example in 
Section~\ref{sec:intro}, discussing the base case $m=1$.  Here we note that 
$F(X_0,0,X_2) \neq 0$ because $H$ contains $P \notin L$. A similar justification
goes along with the second vertical arrow.

Combining both estimates, we find that
\[
  \card{V(F)^{\aff}\setminus \{P\}} = \card{V(F)^{\aff}} - 1 \leq \frac{d}{a_2}-1+q\left(\frac{d}{a_1}-1\right). 
\]
Since  $a_1 < a_2$ and $a_1$ and $a_2$ are coprime  it follows that
\begin{align*}
\card{V(F)}&\leq \frac{d}{a_2}+q\left(\frac{d}{a_1}-1\right)+\card{V(F)\cap H_\infty}\\
           & \hspace{4cm} \downarrow \text{\footnotesize{$X_0 = 0 \rightsquigarrow \PP(a_1,a_2)$}} \\ 
           &\leq \frac{d}{a_2}+q\left(\frac{d}{a_1}-1\right)  \hspace{0.28cm} + \hspace{0.28cm} \frac{d}{a_1a_2}\\
           &=    q\frac{d}{a_1}+1 + \frac{d}{a_2}\frac{a_1+1}{a_1}-q-1\\
           &\leq q\frac{d}{a_1}+1 + \frac{d}{a_1}-q-1\\
           &\leq q\frac{d}{a_1}+1,
\end{align*}
where the last inequality uses our assumption that $d \leq a_1(q+1)$. This ends
the proof in Case 2.\\

\noindent 
\textcolor{blue}{\textbf{Case 3}}: 
One has $P_\infty \in V(F)\setminus L$. This case is similar but easier. Using
the same definition of $X$ with $P = P_\infty$, one finds on the one hand that
\begin{align*}
\card{X} &= \sum_{P'\in V(F)\setminus \{P\}} \card{\set{L}{L \mbox{ is a line with }P,P'\in L}}\\
         &= \sum_{P'\in V(F)\setminus \{P\}} 1                                                 \\
         &=  \card{V(F)} - 1,
\end{align*}
and, on the other hand, that
\begin{align*}
\card{X} &=   \sum_{H \mbox{ \scriptsize{type} } 0 }(\card{V(F)\cap H}- 1)+\sum_{H \mbox{ \scriptsize{type} } 1 }(\card{V(F)\cap H}- 1) \\   
         & \hspace{2.2cm} \downarrow \text{\footnotesize{$X_0 = 0 \rightsquigarrow \PP(a_1,a_2)$}} \hspace{1.1cm} \downarrow \text{\footnotesize{$X_1 = 0 \rightsquigarrow \PP(1,a_2)$}}\\
         &\leq \hspace{0.6cm} 1 \cdot \left(\frac{d}{a_1a_2}-1 \right) \hspace{0.85cm} + \hspace{0.85cm} q\left(\frac{d}{a_2}-1\right).
\end{align*}
Together, this combines to yield
\begin{align*}
\card{V(F)} &\leq   \frac{d}{a_1a_2} + q \left(\frac{d}{a_2} - 1 \right) \\
            &\leq   \frac{d}{a_1}    + q \frac{d}{a_1} - q               \\
            &\leq q \frac{d}{a_1}    + 1 ,
\end{align*}
where the last step uses $d \leq a_1(q+1)$. Thus Theorem~\ref{serreforplanes} 
is proved.

\section{Weighted projective Reed--Muller codes}  
\label{section_WPRM}

In this section, we outline how the considerations of the previous sections can 
be applied to coding theory. Recall that a ($q$-ary) linear code of length $n$
and dimension $k$ is, by definition, a $k$-dimensional subspace of $\Fq^n$. The
minimum distance of such a code $C$ is defined by 
\[
d(C):= \min \set{\wt(x)}{x\in C \text{ with } x\ne 0},
\]
where for any $x= (x_1, \dots , x_n)$, the Hamming weight $\wt(x)$ is the
number of nonzero coordinates in $x$, i.e., $\card{\set{ i }{ x_i \ne 0 }}$. We usually 
say that a $q$-ary linear code $C$ has parameters $[n,k,d]$ or that $C$ is a 
$[n,k,d]_q$-code if $C$ has length $n$, dimension $k$, and minimum distance~$d$.
We shall begin by reviewing some classical families of linear codes. 

\subsection{Generalized Reed--Muller codes, projective Reed--Muller codes and
            projective nested cartesian codes}

The generalized Reed--Muller code over $\Fq$ of order $d$ and with $m$
variables has been introduced by Delsarte, Goethals and MacWilliams in 1970 
in~\cite{DGMcW}. It is denoted by $\RM_q(d,m)$ and defined as the image of the
evaluation map 
\[
c\colon\Fq[X_1,\ldots,X_m]_{\le d} \longrightarrow \Fq^{q^m} 
\quad \text{given by} \quad c(f)=(f(P))_{P\in \AA^m(\Fq)},
\]
where $\Fq[X_1,\ldots,X_m]_{\le d}$ denotes the $\Fq$-vector space of all 
polynomials in $m$ variables $X_1, \dots , X_m$ with coefficients in $\Fq$ and 
with degree $\le d$. 

If $d < q$, then the evaluation map $c$ is injective, and so the dimension of 
$\RM_q(d,m)$ equals $\dim_{\Fq} \Fq[X_1,\ldots,X_m]_{\le d}$, which is
${d+m}\choose{m}$. The minimum distance can be deduced from a classical result 
of Ore (cf.~noted in \cite[Thm. 6.13]{LN}), which implies that the maximal
number of zeros in $\AA^m(\Fq)$ of a polynomial in $\Fq[X_1, \dots , X_m]$ of
degree $d$ is equal to $dq^{m-1}$. Thus we have:

\begin{proposition}
\label{parameters_RM_q(d,m)}
If $d < q$, then the code $\RM_q(d,m)$ has parameters 
\[
\left[q^m, \; {{d+m}\choose{d}}, \; (q-d)q^{m-1} \right].
\]
\end{proposition}

\bigskip

The projective Reed--Muller codes were introduced and studied by 
Lachaud~\cite{lachaud, L} and S{\o}rensen~\cite{So} by the late 1980's and 
early 1990's. They can be defined as follows. 

Choose representatives in $\Fq^{m+1}$ for $\Fq$-rational points of the (usual) 
projective space $\PP^m$  in such a way that the first nonzero coordinate is~$1$.
Let $P_1, \dots , P_{p_m}$ be a fixed collection of such representatives for the 
points of $\PP^m(\Fq)$. Now the evaluation map 
\[
c\colon\Fq[X_1,\ldots,X_m]_{d} \longrightarrow \Fq^{p_m} 
\quad \text{given by} \quad c(f)= \left( f(P_1), \dots f(P_{p_m}) \right)
\]
is injective if $d\le q$ and we define $\PRM_q(d,m)$ to be the image of this map.
Using~\eqref{serreequality}, we can deduce the following.

\begin{proposition}
\label{parameters_PRM_q(d,m)}
If $d\le q$, then the code $\PRM_q(d,m)$ has parameters 
\[
\left[ p_m, \; {{d+m}\choose{d}}, \; (q-d+1)q^{m-1} \right].
\]
\end{proposition}

This construction has been generalized in \cite{A} where the evaluation of the 
homogeneous polynomials is done on the rational points of an hypersurface 
of~$\PP^m(\Fq)$, most notably on quadric hypersurfaces. The parameters of such 
codes have been improved in  3 and 4-dimensional projective spaces in a series 
of papers (see, for example,~\cite{E}).
\bigskip

Recently, Carvalho, Lopez Neumann and L\'opez have proposed in \cite{CLNL} 
another generalization of $\PRM_q(d,m)$. In their paper, the evaluation of
homogeneous polynomials is done  on suitable representatives in $\Fq^{m+1}$ of 
projective cartesian sets 
$
\{(a_0:a_1: \cdots:a_m)\in\PP^m (\Fq) : 
    a_i\in A_i \text{ for } i=0,1, \dots , m\},
$ 
where $A_0, A_1, \ldots,A_m$ are nonempty subsets of $\Fq$. 

%%%%%%%%%%%%%%%%%%%%%%%%%%%%%%%%

\subsection{Weighted projective Reed--Muller codes}
Let $a_0,\ldots,a_m$ be positive integers such that 
$\gcd(a_0, a_1, \dots , a_m)=1$. Denote the $(m+1)$-tuple 
$(a_0, a_1, \dots , a_m)$ by $\w$. Consider an integer $d$ which is a multiple
of the least common multiple of the $a_i$'s, say $d=k\lcm(a_0\ldots a_m)$.

We consider the weighted projective space $\PP(\w) = \PP(a_0,\ldots,a_m)$ 
of dimension $m$ with weights $a_0,\ldots,a_m$ over $\Fq$, whose definition was
recalled in Section \ref{sec:intro}. Note that  $\PP(a_0,\ldots,a_m)$ is a 
disjoint union of $W_0, W_1, \dots , W_m$, where for $0\le i \le m$, 
\[
W_i:= \set{ (x_0:\cdots:x_m)\in \PP(a_0,\ldots,a_m) }{ x_0=\cdots=x_{i-1}=0, \ x_i \ne 0 }.
\]

As before, let $\sansS_d$ denote the space of weighted homogeneous polynomials 
of degree $d$. We define the Weighted Projective Reed--Muller code of order $d$
over  $\PP(a_0,\ldots,a_m)(\Fq)$, denoted by $\WPRM_q(d,m; \w)$, as the image 
of the linear map
\[
c\colon \sansS_d \longrightarrow \Fq^{p_m} 
\quad \text{given by} \quad c(F)=(c_x(F))_{x\in \PP(\w)(\Fq)},
\]
where for $x = (x_0: x_1: \cdots:x_m) \in  \PP(\w)(\Fq)$,
\[
c_x(F)=\frac{F(x_0,\ldots,x_m)}{x_i^{d/a_i}} \quad \text{if} \ x=(x_0:\cdots:x_m)\in W_i.
\]
Observe that the map $c$ is well defined. Indeed, for a nonzero 
$\lambda \in \overline{\FF}_q$, if 
$y=(\lambda^{a_0} x_0:\cdots :\lambda^{a_m} x_m)=(x_0:\cdots:x_m)=x \in W_i$,
then 
\[
c_y(F) = \frac{F(\lambda^{a_0}x_0,\ldots,\lambda^{a_m}x_m)}{(\lambda^{a_i}x_i)^{d/a_i}}
       = \frac{\lambda^dF(x_0,\ldots,x_m)}{\lambda^dx_i^{d/a_i}}
       = c_x(F).
\]
This argument shows also that $c_x(F)\in \Fq$ since every point $x$ of 
$\PP(\w)(\Fq)$ has weighted homogeneous coordinates $(x_0:x_1: \cdots:x_m)$
such that $x_i\in \Fq$ for $i=0, 1, \dots, m$. 

\subsubsection{Length and dimension}

The length of $\WPRM_q(d,m; \w)$ is clearly $p_m=q^m+\cdots+q+1$. Assume that
$d \le q$. Then the  linear map $c$ in injective and so the dimension of
$\WPRM_q(d,m; \w)$ is equal to the dimension of the $\Fq$-vector space 
$\sansS_d$, which is equal to the number of representations of $d$ as a 
nonnegative integer linear combination of $a_0,\ldots,a_m$:
\[
\left|  \{ (\alpha_0,\ldots,\alpha_m)\in{\ZZ}^{m+1}_{\ge 0} : 
                  \alpha_0a_0+\cdots +\alpha_ma_m=d\} \right|.
\]
Note that, using a theorem of Schur (see, e.g., \cite[Thm. 3.15.2]{WZ}), we
have an asymptotic formula
\[
\dim \WPRM_q(d,m; \w) =\frac{d^{m}}{m!a_0\ldots a_m} + O(d^{m-1}) 
                       \quad \text{when } d\rightarrow \infty. 
\]

If we suppose that $a_0=1$, then this dimension is equal to
\[
\left| \{(\alpha_1,\ldots,\alpha_m)\in{\ZZ}^{m+1}_{\ge 0}  : 
               \alpha_1a_1+\cdots +\alpha_ma_m\leq d\} \right|.
 \]
This can be viewed as the number of integral points in an integral convex
polytope and then the dimension can be obtained using Ehrhart polynomials 
(see the examples below in dimension 2).

\subsubsection{Minimum distance}
The minimum distance of  $\WPRM_q(d,m; \w)$ is equal to the number of rational
points on $\PP(\w)(\Fq)$ minus the maximal number of points on a hypersurface 
$V$ of degree $d$ of $\PP(\w)(\Fq)$. Thus we can determine it using the results
of the previous sections. 
 
First, suppose $i,j \in \{0,1, \dots , m\}$ and $d'\in \ZZ$ are such that 
\[
\lcm(a_i,a_j)=\min\{\lcm(a_r,a_s), r\not=s\}, 
                \quad \text{and} \quad d':= \frac{d}{\lcm(a_i,a_j)}.
\]
Then from Lemma~\ref{serrelowerbound}, we see that 
\[
d(\WPRM_q(d,m; \w)) \le (q-d'+1)q^{m-1}.
\]
Furthermore, if $a_0=1$ and $m=2$ and we assume, without loss of generality 
that $a_1 \le a_2$, then from Theorem \ref{serreforplanes}, we see that 
\begin{equation}
\label{exactmindist}
d(\WPRM_q(d, 2; \w)) = \left( q-\frac{d}{a_1}+1 \right) q^{m-1}.
\end{equation}

\subsubsection{A particular case}

Consider the particular case of the weighted projective plane $\PP(1,1,a)$, 
where $a$ is a positive integer. Also let $\w = (1,1,a)$. Given a convex
polytope $\Delta$ whose vertices have integral coordinates, the function which 
assigns to a nonnegative integer $k$ the number $\card{k \Delta\cap{\ZZ}^m}$ of 
integral points in dilates $k\Delta$ of  $\Delta$ is a polynomial of degree~$m$,
called the Ehrhart polynomial of $\Delta$ (see, for example,~\cite{BDLDPS}).
For $m=2$, this polynomial can be written in the following way:
\[
\card{ k\Delta\cap{\ZZ}^2 } =\Vol(\Delta)k^2+\frac{1}{2} \card{\partial\Delta\cap{\ZZ}^2} k +1.
\]
Hence we find that, for $d=ka$, the dimension of the code $\WPRM_q(d,2; \w)$ 
is equal~to 
\[
\frac{1}{2}ak^2+\frac{a+2}{2} k +1
  =\frac{d^2}{2a}+\frac{(a+2)d}{2a}+1
  =\frac{(d+a)(d+2)}{2a}.
\]
Since we have $d'=d$ in our case, we find from \eqref{exactmindist} that the 
minimum distance of $\WPRM_q(d,2; \w)$ is $q^2-(d-1)q$.

Thus, the code $\WPRM_q(d,2; \w)$ has parameters 
\[
\left[p_2, \; \frac{(d+a)(d+2)}{2a},  \; q^2-(d-1)q \right]
\]
 and we can compare it to the parameters of the code $\PRM_q(d,2)$, which are
\[
\left[ p_2, \; \frac{(d+1)(d+2)}{2}, \; q^2-(d-1)q \right].
\]
We find here that the weighted projective Reed--Muller code has the same length
and the same minimum distance, but worse dimension than the projective
Reed--Muller code.

\subsubsection{Another particular case}

Let $a, b$ be positive integers with $a\le b$ and let $\w = (1,a,b)$. Consider 
the particular case of the weighted projective plane $\PP^2(1,a,b)$ and 
consider an integer $d=k\lcm(a,b)$ with $d\leq q$. Arguing as before, we can 
deduce the following. 

\begin{proposition}
\label{parameters_WPRM_q(d,2)_{(1,a,b)}}
The code $\WPRM_q(d,2; \w)$ has parameters 
\[
\left[ p_2,  \; \frac{(d+2a)(d+b)+(\gcd(a,b)-a)d}{2ab},  \; q^2- \left(\frac{d}{a}-1 \right)q \right].
\]
\end{proposition}

In particular, if $a=2$ and $b\ge 2$, we see that the minimum distance of the
code $\WPRM_q(d,2; (1,2,b))$  is always better than the minimum distance of 
$\PRM_q(d,2)$, but the dimension of $\WPRM_q(d,2; (1,2,b))$  is always worse 
than the dimension of $\PRM_q(d,2)$.

\subsubsection{Relative parameters}

Recall that, for any code $C$, the transmission rate $R(C)$ and the relative 
distance $\delta(C)$ of $C$ are defined by 
\[
R(C) =\frac{\dim C}{\length C}  \quad \text{and} \quad  \delta (C)=\frac{\dist C}{\length C}.
\]
The number
\[
\lambda(C)=R(C)+\delta(C)=(\dim C + \dist C)/\length C
\]
is a parameter of $C$ and it is suggested in \cite{L} that it can be taken as a
measure of the performance of the code $C$.

It is proved in \cite{L} that if $q\geq d+1$, $m\geq 2$, and $d\geq 2m/(m-1)$,
then
\[
\lambda(\PRM_q(d,m)) > \lambda(\RM_q(d,m)).
\]

If $q$ is sufficiently large then one can show that $\WPRM_q(d,2; (1,2,2))$ has
a greater (and thus better) performance than $\PRM_q(d,2)$:

\begin{proposition}
If $q\geq \frac{3k+3}{2}$, then 
\[
\lambda(\WPRM_q(2k,2; (1,2,2))) \geq \lambda(\PRM_q(2k,2)).
\]
\end{proposition}

\begin{proof}
Since the lengths of these codes are equal (namely to $p_2$), we just have to
show that the sum of the dimension and the minimum distance is greater for the
first code when $q$ is sufficiently large. Applying  
Propositions~\ref{parameters_PRM_q(d,m)} and~\ref{parameters_WPRM_q(d,2)_{(1,a,b)}}
yields the desired result.\qed
\end{proof}

In the same way, it is easy to see that:

\begin{proposition}
If $q\geq \frac{7k+4}{2}$, then 
\[
\lambda(\WPRM_q(4k, \; 2; \; (1,2,4))) \geq \lambda(\PRM_q(4k,2)).
\]
\end{proposition}

More generally, using Propositions \ref{parameters_PRM_q(d,m)} and  
\ref{parameters_WPRM_q(d,2)_{(1,a,b)}} we can show that:

\begin{theorem}
For any nonnegative integers $a, \beta$ and $k$ with $a\ge 2$, 
\[
\lambda(\WPRM_q(ka\beta, \; 2; \; (1,a,a\beta))) \geq \lambda(\PRM_q(ka\beta,2)),
\]
provided
\[ 
q\geq \frac{k\beta^2a^2+3\beta a-k\beta -\beta-2}{2\beta(a-1)}.
\]
\end{theorem}

Let us compare the performance over ${\FF}_{19}$ and in degree 16 of the 
generalized Reed--Muller code over ${\AA}^2$, the projective Reed--Muller 
code over $\PP^2$, and the weighted projective Reed--Muller codes over the 
five different weighted projective planes $\PP(1,2,2)$,  $\PP(1,2,4)$,  
$\PP(1,2,8)$,  $\PP(1,4,4)$ and $\PP(1,16,16)$.
 
We find that $\RM_{19}(16,2)$ has parameters $[361, 153, 57]$ and the projective 
counterpart $\PRM_{19}(16,2)$ has parameters $[381,153,76]$, whereas 
\begin{itemize}
\item
$\WPRM_{19}(16,2; (1,2,2))$ has parameters $[381,45,228]$,
\item
$\WPRM_{19}(16,2; (1,2,4))$ has parameters $[381,25,228]$,
\item
$\WPRM_{19}(16,2; (1,2,8))$ has parameters $[381,15,228]$, 
\item
$\WPRM_{19}(16,2; (1,4,4))$ has parameters $[381,15,304]$,
and
\item
$\WPRM_{19}(16,2; (1,16,16))$ has parameters $[381,3,361]$.
\end{itemize}
The affine and projective Reed--Muller codes above have performances
\[
\lambda(\RM_{19}(16,2))=0.581... \quad \text{and} \quad \lambda(\PRM_{19}(16,2))=0.601..., 
\]
whereas the performances of the above five weighted projective Reed--Muller codes are 
$0.716...$, $0.664...$, $0.637...$,  $0.837...$, and $0.955...$ respectively. 

%%%%%%%

\section*{Acknowledgement}
This work was initiated during a week-long IPAM workshop on Algebraic Geometry 
for Coding and Cryptography, that was held in UCLA during February 2016. 
The authors would like to thank the organizers of the workshop, namely, Everett Howe,
Kristin Lauter and Judy Walker for giving them this opportunity, and the anonymous
referee for various helpful comments.
The second author was partially supported by the European Commission under the ICT programme
with contract H2020-ICT-2014-1 645622 PQCRYPTO, and through the European Research
Council under the FP7/2007-2013 programme with ERC Grant Agreement 615722 MOTMELSUM.

\bibliographystyle{hplaindoi}
\bibliography{main}

%%%%%%%%%%%%%%%%%%%%%%%%%%%%%%%%%%%%%%%

\begin{subappendices}

\section{Appendix: Weighted projective spaces}
This appendix is aimed at providing a handy reference for weighted projective 
spaces over arbitrary fields. While some proofs are occasionally outlined, for 
most part we provide complete statements of results and suitable references 
where proofs can be found.

\subsection{Definitions of weighted projective spaces}

\subsubsection{WPS as a Proj functor}

Let $k$ be a field and let $\w = (a_{0}, \dots, a_{m})$ be a sequence of 
strictly positive integers. The condition
\[
\deg X_{i} = a_{i}, \qquad i = 0, \dots, m
\]
defines a gradation of type $\ZZ$ on the polynomial algebra 
$\sansS = k[X_{0}, \dots, X_{m}]$:
\[
\sansS = \bigoplus_{n \geq 0} \sansS_{n}
\]
such that $\sansS_{n} = 0$ if $n < 0$.
%\[
%\text{If} \quad
%f = X_{0}^{r_{0}} \dots X_{m}^{r_{m}} \quad \text{then} \quad
%\wdeg f = n \quad \Longleftrightarrow \quad a_{0} r_{0} + \dots + a_{m} r_{m} = n
%\]
%%%% NEW VERSION:
For a monomial $f = X_{0}^{r_{0}} \dots X_{m}^{r_{m}}$, we have
\[
\wdeg f = n \quad \Longleftrightarrow \quad a_{0} r_{0} + \dots + a_{m} r_{m} = n.
\]
%%%% END NEW VERSION
We assume that the characteristic $p$ of $k$ is coprime to all 
$a_{i}$ ($0\le i \le m$), and that $\gcd(a_{0}, \dots, a_{m}) = 1$. The 
\emph{weighted projective space} (WPS) with sequence of weights $\w$ over $k$ 
is the scheme $\PP(\w) = \Proj \sansS(\w)$. If $\w = (1, \dots, 1)$, we recover
the usual projective space:
\[
\PP(1, \dots, 1) = \PP^{m}.
\]

\subsubsection{Quotients}

Let $G$ be an affine algebraic group over a field $k$ acting on an algebraic 
variety $X$ over $k$. A \emph{categorical quotient} of $X$ by $G$, 
see \cite[p.~92]{Dgv2003},  \cite[Ch.~V, \S 1]{SGA1}, \cite[Def.~0.5, p.~3]{GIT},
is a morphism $p \colon X \longrightarrow Y$, where $Y$ is a variety over~$k$, 
such that
\begin{enumerate}
\item \label{Qsur}
      $\mathsf{p}$ is surjective.
\item \label{Qcat}
      $\mathsf{p}$ is $G$-invariant, that is, $G$-equivariant, that is, 
      $\mathsf{p}$ is constant on the orbits of $G$.  
\item \label{Quniv}
      If $f\colon X \longrightarrow Z$ is a $k$-morphism constant on the orbits
      of~$G$, then there exists a $k$-morphism $\varphi\colon Y \longrightarrow Z$
      such that $f = \varphi \circ \mathsf{p}$.
      \[
      \xymatrix{
      X \ar[d]_{\mathsf{p}} \ar[rd]^{f} \\
      Y \ar[r]_{\varphi} & Z}
      \]
\end{enumerate}
The couple $(Y,\mathsf{p})$ is unique up to unique isomorphism.  A categorical
quotient is called a \emph{geometric quotient}, 
see \cite[p.~92]{Dgv2003}, \cite[Def.~0.6, p.~4]{GIT}, if moreover
\begin{enumerate}
\setcounter{enumi}{3}
\item \label{Qsub}
      $\mathsf{p}$ is open. 
\item \label{Qgeo}
      The fibres of $\mathsf{p}$ are the orbits of $G$ in $X$.
\end{enumerate}

\subsubsection{WPS as a quotient of the punctured affine space}

The gradation $\w$ of $\sansS$ defines an action
\[
\begin{tikzcd}
\sigma \colon \GG_{m} \times \Aff^{m + 1} \arrow[r] & \Aff^{m + 1}
\end{tikzcd}
\]
of $\GG_{m}$ on $\Aff^{m + 1}$ such that
\[
\sigma(t).(x_{0}, \dots, x_{m}) = t.(x_{0}, \dots, x_{m}) = (t^{a_0}x_{0}, \dots,  t^{a_m} x_{m}).
\]
The corresponding morphism
\[
\begin{tikzcd}
\sigma^{\flat} \colon \sansS \arrow[r] & k[T, T^{-1}] \otimes \sansS \simeq \sansS[T, T^{-1}]
\end{tikzcd}
\]
is such that
\[
[\sigma^{\flat}f](T, X_{0}, \dots, X_{m}) = f(T^{a_{0}} X_{0}, \dots, T^{a_{m}} X_{m}).
\]
The algebra $\sansS[T, T^{-1}]$ is called the algebra of \emph{Laurent polynomials}
over~$\sansS$. The group $\GG_{m}$ operates as well on the pointed cone
\[
\VV = \Aff^{m + 1} \setminus \{0\}.
\]

\goodbreak

\begin{theorem}
\label{projgq}
The morphism
\[
\begin{tikzcd}
p \colon \VV \arrow[r] & \VV/\GG_{m}
\end{tikzcd}
\]
is a geometric quotient, and there is an isomorphism
\[
\begin{tikzcd}
\iota \colon \VV/ \GG_{m} \arrow[r,"\sim"] & \PP(\w)
\end{tikzcd}
\]
\end{theorem}

\begin{proof}
\cite[1.21, p. 36]{Dgv1982}, \cite[Ex. 6.2, p. 96]{Dgv2003}.\qed
\end{proof}

The scheme $\PP(\w)$ is a normal irreducible projective variety, of dimension~$m$ 
\cite[p.~5]{GIT}, \cite[1.3.3]{Dgv1982}.

\subsubsection{WPS as a finite quotient of the projective space}

For any integer $n > 0$, we denote by $\bmu_{n}$ the finite group scheme of 
$n$-th roots of unity, with coordinate ring $k[X]/(X^{n} - 1)$. We put
\[
G = G_{\w} = \bmu_{a_{0}} \times \dots \times \bmu_{a_{m}}.
\]
Then $\card{G_{\w}} = a$, with $a = a_{0} \dots a_{m}$, and $G_{\w} \simeq \bmu_{a}$
if and only if $a$ is the l.c.m. of $a_{0}, \dots, a_{m}$, that is, if and only 
if $a_{0}, \dots, a_{m}$ are pairwise coprime. There is a linear action of $G$ 
on $\PP^{m}$ given by
\[
(\zeta_{0}, \dots, \zeta_{m}).(x_{0} : \ldots : x_{m}) 
   = (\zeta_{0} x_{0} : \ldots : \zeta_{m} x_{m})
\]
The morphism $\pi_{0}\colon \VV \to \VV$ given by
\[
\pi_{0}(x_{0}, \dots, x_{m}) = (x_{0}^{a_0}, \dots, x_{m}^{a_m})
\]
induces a diagram
\[
\xymatrix{
\VV \ar[d]^{p} \ar[rr]^{\pi_{0}}  &                          & \VV\ar[d]^{p} \\
\PP^{m} \ar[rr]^{\pi} \ar[rd]^{p} &                          & \PP(\w)       \\
                                  & \PP^{m}/G \ar[ru]^{\sim}                  } 
\]
Let $G$ be an affine algebraic group over a field $k$ acting on an algebraic
variety $X$ over~$k$. For the definition of a \emph{good geometric quotient} of
$X$ by~$G$, see \cite[p. 92]{Dgv2003}. We denote by $G(x)$ the \emph{stabilizer}
or \emph{isotropy group} of~$X$. The action is \emph{free} at $x$ if $G(x)$ is
trivial.

\goodbreak

\begin{proposition}
\label{GGQ}
The morphism $\pi \colon  \PP^{m} \longrightarrow \PP(\w)$ given by
\[
\pi(x_{0} : \ldots : x_{m}) = (x_{0}^{a_0}: \ldots : x_{m}^{a_m})
\]
is a good geometric quotient of $X$ by $G$, and therefore enjoys the following 
properties: 
\begin{enumerate}
\item $\pi$ is surjective, finite and submersive.
\item The fibres of $\pi$ are the orbits of $G$ in $\PP^{m}$.
\item If $x \in \PP^{m}$ and $y = \pi(x) \in \PP(\w)$, the residual field 
     $\kappa(x)$ is a Galois extension of $\kappa(y)$ and the canonical 
     homomorphism of $G(x)$ in the group $\Gal(\kappa(x)/\kappa(y))$ of 
     $\kappa(y)$-automorphisms of $\kappa(x)$ is surjective.
\end{enumerate}
\end{proposition}

\begin{proof}
See \cite[Ch.~III, Prop.~19]{Serre1988},  \cite[Ch.~V, Props.~1.3 and 1.8]{SGA1},
\cite[Ex.~6.1, p.~95]{Dgv2003}.\qed
\end{proof}

Notice that $\deg \pi = a_0 \ldots a_m$. The Jacobian matrix of $\pi$ is
\[
d\pi(x) = \Diag(a_0 x_{0}^{a_0 - 1}, \ldots, a_m x_{m}^{a_m - 1}),
\]
and
\[
\det d\pi(x) = (a_{0} \ldots a_{m}) x_{0}^{a_0 - 1} \dots x_{m}^{a_m - 1}
\]
If we denote by $H_{i}$ the hyperplane $x_{i} = 0$, the ramification locus is
\[
R = \bigcup_{a_{i} > 1} H_{i}.
\]
Then $\pi$ is \'etale outside $R$, which clearly contains the singular set.

\begin{proposition}
The scheme $\PP(\w)$ is Cohen--Macaulay.
\end{proposition}

\begin{proof}
Cf.~\cite[Th.~3A.1]{B-M}.\qed
\end{proof}

\subsection{The singular locus}

We say that the sequence of weights $\w$ is \emph{normalized} \cite[p.~185]{Dimca}
or \emph{well formed} \cite[Def.~3.3.4]{Hosgood} if
\[
\gcd(a_{0}, \dots, \widehat{a_{i}}, \dots, a_{m}) = 1 
      \quad \text{for every} \ 0 \leq i \leq m.
\]
Any weighted projective space is isomorphic to a well-formed weighted 
projective space [\emph{loc.~cit}]. If $p$ is a prime number, we put
\[
I(p) = \set{i \in \{1, \dots, m \}}{p \ \text{divides} \  a_{i}}.
\]
The set $\Sigma = \Sigma(\w)$ of prime numbers such that $I(p) \neq \emptyset$
is finite, and $\w$ is \emph{normalized} if and only if $\card{I(p)} \leq m - 1$
for every $p$. The space
\[
S(p) = \set{x \in \PP(\w)}{x_{i} = 0  \ \text{if} \ i \notin I(p)}
\]
is a weighted projective space of dimension $\card{I(p)}$.

\begin{proposition}
\label{Dimca}
Assume that $\w$ is normalized.
\begin{enumerate}
\item The decomposition of $\Sing \PP(\w)$ into irreducible components is
      \[
      \Sing \PP(\w) = \bigcup_{p \in \Sigma} S(p).
      \]
\item Moreover
      \[
      \Sing \PP(\w) = \set{x \in \Sing \PP(\w)}{G_{x} \neq \{1\}}.
      \]
\end{enumerate}
\end{proposition}

\begin{proof}
See Dimca \cite[p.~185]{Dimca}.\qed
\end{proof}

Notice that $\dim \Sing \PP(\w) \leq m - 2$, that is, $\PP(\w)$ is regular in
codimension one, as it already follows from normality.

\begin{corollary}
Assume that $\w$ is normalized.
\begin{enumerate}
\item \label{RAM1}
      If $(x_{0} : \ldots : x_{m}) \in \Sing \PP(\w)$, then  $x_{i} = 0$ 
      for at least one $i$.
\item \label{RAM2}
      If
      \[
      \gcd(a_{j}, a_{j}) = 1 \quad \text{for every couple} \ (i, j)  \ \text{with} \ j \neq i,
      \]
      then
      \[
      \Sing \PP(\w) = \{P_{0}, \dots, P_{m}\},
      \]
      where $P_{i}$ are the $m + 1$ vertices $(0 : \ldots : 1 : \ldots  :  0)$.
\end{enumerate}
\end{corollary}

\begin{proof}
From Proposition~\ref{Dimca} we deduce that if $x \in \Sing \PP(\w)$, then
$x \in S(p)$ for some $p \in \Sigma$, hence, $x_{i} = 0$ for at least one~$i$.
This proves~\eqref{RAM1}. If $a_{0}, \dots, a_{m}$ are pairwise coprime, then 
$I(p)$ has only one element $i$, and $S(p) = \{ P_{i} \}$. This 
proves~\eqref{RAM2}. \qed
\end{proof}

\subsection{Affine parts}

\subsubsection{Quotient of the affine space by a cyclic group}
\label{QAC}

We shall define an action of the cyclic group $\bmu_{a_{i}}$ on $\Aff^{m}$, 
which is called \emph{the action of type} 
\[
\frac{1}{a_{i}}(a_{0}, \dots, \widehat{a_{i}}, \dots, a_{m}).
\]
Let $\Aff_{\{i\}}^{m}$ the affine hypersurface of $\VV$ with equation $X_{i} = 1$.
Our action is defined by
\[
\zeta.(x_{0}, \dots, 1, \dots, x_{m}) =
   (\zeta^{a_{0}} x_{0}, \dots, 1, \dots, \zeta^{a_{m}} x_{m}), \quad \zeta \in \bmu_{a_{i}},
\]
and we get a finite quotient
\[
\begin{tikzcd}
p \colon \Aff_{\{i\}}^{m} \arrow[r] & \Aff_{\{i\}}^{m}/\bmu_{a_{i}}.
\end{tikzcd}
\]
We have
\[
k[\Aff_{\{i\}}^{m}] = \sansS /(X_{i} - 1) = k[X_{0}, \dots, \widehat{X_{i}}, \dots, X_{m}].
\]
If
\[
k[\Aff_{\{i\}}^{m}]^{\inv} = k[\Aff_{\{i\}}^{m}/\bmu_{a_{i}}] 
                           = k[\Aff_{\{i\}}^{m}]^{\bmu_{a_{i}}},
\]
then \cite[Lem. 2.5, p. 11]{B-M}
\[
k[\Aff^{m}_{\{i\}}]^{\inv} = \bigoplus k[\Aff_{\{i\}}^{m}]_{n a_{i}}.
\]
If $\gcd(a_{j}, a_{i}) = 1$ for $j \neq i$, then the only point $x \in \Aff^{m}$
with non-trivial isotropy subgroup is $x = 0$, and the projection 
$\Aff_{\{i\}}^{m} \to \Aff_{\{i\}}^{m}/\bmu_{a_{i}}$ is \'etale outside $0$.

\subsubsection{Affine parts}

For $0 \leq i \leq m$, we consider the principal open subset
\[
V_{i} = \set{x \in \VV}{x_{i} \neq 0}.
\]
Then $k[V_{i}]$ is the localization of $\sansS$ with respect to $X_{i}$, namely
\[
k[V_{i}] = k\left[\frac{1}{X_{i}}\right] =
     \set{\frac{f}{X_{i}^{n}}}{f \in \sansS} \subset k(\Aff^{m + 1}).
\]
We put
\[
U_{i} = p(V_{i}) = V_{i}/\GG_{m} = \set{x = (x_0:\ldots:x_n)\in \PP(\w)}{x_{i} \neq 0}
\]
and we consider the $k$-subalgebra of degree $0$ elements of $k[V_{i}]$:
\begin{equation}
\label{quotalg}
k[V_{i}]^{0} =
    \set{\frac{f}{X_{i}^{n}}\in \sansS_{[i]}}{f \ \text{homogeneous}, \ n \geq 0, \ \deg f = n a_{i}}.
\end{equation}
Then
\[
k[U_{i}] = k[V_{i}]^{0} = k[V_{i}]^{\GG_{m}}.
\]

\begin{proposition}
\label{pieces}
With the preceding notation\textup{:}
\begin{enumerate}
\item The projection $p \colon \Aff_{\{i\}}^{m} \rightarrow U_{i}$ given by
      \[
      p(x_{0}, \dots,  1,  \dots, x_{m}) = (x_{0} : \ldots : 1 : \ldots  :  x_{m})
      \]
      is surjective and induces an isomorphism
      \[
      \begin{tikzcd}
      \varphi \colon \Aff_{\{i\}}^{m}/\bmu_{a_{i}}  \arrow[r,"\sim"] & U_{i},
      \end{tikzcd}
      \]
      with an inverse
      \[
      \begin{tikzcd}
      \psi \colon U_{i} \arrow[r,"\sim"] & \Aff_{\{i\}}^{m}/\bmu_{a_{i}}
      \end{tikzcd}
      \]
      such that
      \[
      \psi(x_{0} : \ldots : 1 : \ldots : x_{m}) = (x_{0}, \dots, 1, \dots, x_{m}).
      \]
\item The canonical homomorphism $p^{\flat} \colon k[U_{i}] \to k[\Aff_{\{i\}}^{m}]$ 
      given by
      \[
      p^{\flat} \left( \frac{f}{X_{i}^{n}} \right) = f(X_{0}, \dots, 1, \dots, X_{m}),
      \]
      for $f$ homogeneous, $n \geq 0$, $\wdeg f = n a_{i}$, is injective and 
      induces an isomorphism
      \[
      \begin{tikzcd}
      \varphi^{\flat} \colon k[U_{i}] \arrow[r, "\sim"] & k[\Aff^{m}_{\{i\}}]^{\inv},
      \end{tikzcd}
      \]
      with an inverse
      \[
      \begin{tikzcd}
      \psi^{\flat} \colon k[\Aff^{m}_{\{i\}}]^{\inv} \arrow[r] & k[U_{i}]
      \end{tikzcd}
      \]
      such that
      \[
      \psi^{\flat}(f) = \frac{f}{X_{i}^{n}},
      \]
      for $f \in k[\Aff^{m}_{\{i\}}]^{\inv}$, $\deg f = n a_{i}$. In particular
      \[
      \psi^{\flat}\left(X_{j}^{a_{i}}\right) = \frac{X_{j}^{a_{i}}}{X_{i}^{a_{j}}} \, .
      \]
\end{enumerate}
\end{proposition}

Proposition \ref{pieces} leads to the two diagrams
\[
\begin{tikzcd}
\Aff_{\{i\}}^{m} \arrow[d]  \arrow[r, "\subset"] \arrow[dr, "p"] & V_{i} \arrow[d, "p"]  \arrow[r, "\subset"] & \VV \arrow[d, "p"] \\
\Aff_{\{i\}}^{m}/\bmu_{a_{i}}  \arrow[r, "\varphi"]              & U_{i} \arrow[r, "\subset"]                  & \PP(\w) 
\end{tikzcd}
\qquad\qquad
\begin{tikzcd}
k[\Aff_{\{i\}}^{m}]                            & k[V_{i}] \arrow[l] & \sansS \arrow[l, "\supset"'] \\
k[\Aff^{m}_{\{i\}}]^{\inv} \arrow[u, "\cup"]
& k[U_{i}] \arrow[u, "\cup"'] \arrow[l, "\sim", "\varphi^{\flat}"'] \arrow[ul, "p^{\flat}"']
\end{tikzcd}
\]

\begin{proof}[Proof of proposition \ref{pieces}]\quad

1. See \cite[Th.~2.6.b, p.~12]{B-M}, \cite[1.2.3]{Hosgood}, and occasionally
see also \cite[pp.~63--64]{Tevelev} and \cite[pp.~4--5]{Reid}.

2. Let $x$ and $y$ be in $\Aff_{\{i\}}^{m}$. If
\[
(y_{0}, \dots, 1, \dots, y_{m}) =
   (\zeta^{a_{0}} x_{0}, \dots, 1, \dots, \zeta^{a_{m}} x_{m}), \quad \zeta \in \bmu_{a_{i}},
\]
then $p(x) = p(y)$, and the existence of $\varphi$ follows. Conversely, assume
that $p(y) = p(x)$. Then we have in $\VV$, with some $t \in \GG_{m}$:
\[
(v_{0}, \ldots, 1,  \ldots, v_{m}) =
      (t^{a_{0}} u_{0}, \ldots, t^{a_{i}},  \ldots, t^{a_{m}} u_{m})
\]
This implies that $t \in \bmu_{a_{i}}$, hence, $p$ factors modulo $\bmu_{a_{i}}$,
and $\varphi$ is injective.

3. Let
\[
W_{i} = \set{x = (x_0 : \ldots : \xi_{i} : \ldots : x_n)  \in \PP(a_{0}, \dots, 1, \dots, a_{m})}
            {\xi_{i} \neq 0}
\]
and consider the morphisms
\[
\begin{tikzcd}
m \colon W_{i} \arrow[r] & U_{i}
\end{tikzcd}
\]
given by
\[
m(x_{0} : \ldots : \xi_{i} : \ldots : x_{m}) = (x_{0} : \ldots : \xi_{i}^{a_{i}} : \ldots : x_{m})
\]
and
\[
\begin{tikzcd}
\psi_{0} \colon W_{i} \arrow[r] & \Aff_{\{i\}}^{m}/\bmu_{a_{i}}
\end{tikzcd}
\]
given by
\[
\psi_{0}(x_{0} : \ldots : \xi_{i} : \ldots : x_{m}) =
      (\frac{x_{0}}{\xi_{i}^{a_{0}}}, \dots, 1, \dots, \frac{x_{m}}{\xi_{i}^{a_{m}}}).
\]
If $m(x) = m(y)$, then $\eta_{i} = t \xi_{i}$ with $t \in \bmu_{a_{i}}$ and 
$\psi_{0}(x) = \psi_{0}(y)$. Hence, there is a morphism
\[
\begin{tikzcd}
\psi \colon U_{i} \arrow[r] & \Aff_{\{i\}}^{m}/\bmu_{a_{i}}
\end{tikzcd}
\]
such that $\psi_{0} = \psi \circ m$:
\[
\begin{tikzcd}
W_{i} \arrow[r, "\psi_{0}"] \arrow[d, "m"] & \Aff_{\{i\}}^{m}/\bmu_{a_{i}} \\
U_{i} \arrow[ur, "\psi"]
\end{tikzcd}
\]
We have
\[
\psi(x_{0} : \ldots : 1 : \ldots : x_{m}) = (x_{0}, \dots, 1, \dots, x_{m}).
\]
This implies $\psi \circ \varphi (x) = x$ if $x \in \Aff_{\{i\}}^{m}/\bmu_{a_{i}}$,
and $\psi$ is surjective. On the other hand, it is easy to see that 
$\varphi \circ \psi \circ m (x) = m(x)$ if $x \in W_{i}$, hence, 
$\varphi \circ \psi(x) = x$ if $x \in U_{i}$, and $\varphi$ is surjective.

4. The corresponding homomorphisms of algebras are respectively
\[
\begin{tikzcd}
m^{\flat} \colon \sansS_{(i)} = k[U_{i}] \arrow[r] & k[W_{i}]
\end{tikzcd}
\]
given by
\[
m^{\flat} (f / X_{i}^{n}) = f / \Xi_{i}^{\, n a_{i}},
\]
for $f$ homogeneous, $n \geq 0$, $\wdeg f = n a_{i}$, and
\[
\begin{tikzcd}
\psi_{0}^{\flat} \colon k[\Aff^{m}_{\{i\}}]^{\inv} \arrow[r] & k[W_{i}]
\end{tikzcd}
\]
given by
\[
\psi_{0}^{\flat}(f) = f / \Xi_{i}^{\, n a_{i}},
\]
for $f \in \mathsf{R}_{\{i\}}$, $\deg f = n a_{i}$. Now the morphism
\[
\begin{tikzcd}
\psi^{\flat} \colon k[\Aff^{m}_{\{i\}}]^{\inv} \arrow[r] & k[U_{i}]
\end{tikzcd}
\]
such that
\[
\psi^{\flat}(f) = f / X_{i}^{n},
\]
for $f \in k[\Aff^{m}_{\{i\}}]^{\inv}$, $\deg f = n a_{i}$, satisfies 
$\psi_{0}^{\flat} = m^{\flat} \circ \psi^{\flat}$, and we have a diagram
\[
\begin{tikzcd}
k[W_{i}] & k[\Aff^{m}_{\{i\}}]^{\inv} \arrow[l, "\psi_{0}^{\flat}"']  \arrow[dl, "\psi^{\flat}"] \\
k[U_{i}] \arrow[u, "m^{\flat}"]
\end{tikzcd}
\]
\qed
\end{proof}

\begin{remark}
Roughly speaking, we have
\[
\psi(x_{0} : \ldots : x_{i} : \ldots : x_{m}) =
\left( \frac{x_{0}}{x_{i}^{a_{0}/a_{i}}}, \dots, 1, \dots,\frac{x_{m}}{x_{i}^{a_{m}/a_{i}}} \right).
\]
This formula obviously makes sense if $a_{i} = 1$ (see below).
\end{remark}

From Proposition \ref{pieces} we get, see also \cite[p.~81]{Kollar} and 
\cite[Prop.~1.3.3(ii)]{Dgv1982}:

\begin{corollary}
\label{projcyclic}
The space $\PP(\w)$ has cyclic quotient singularities.
\end{corollary}

Similarly, if $k = \RR$, the space $\PP(\w)$ is an orbifold (or $V$-\emph{variety})
\cite[Th.~3.1.6]{Dgv1982}.

\subsubsection{A special case}

The complement of the open set $U_{i}$ is the hyperplane $P_{i}$ of $\PP(\w)$
with equation $x_{i} = 0$. Then $P_{i}$ is the weighted projective space
$\PP(\w')$ of dimension $m - 1$, with 
$\w' = (a_{0}, \dots, \widehat{a_{i}}, \dots, a_{m})$, and we have the 
standard ``motivic'' decomposition
\begin{equation}
\label{motivic}
\PP(\w) = U_{i} \sqcup P_{i}.
\end{equation}
If we assume $\w = (a_{0}, \dots, 1, \dots, a_{m})$, with $a_{i} = 1$, the 
set $U_{i}$ is affine, since  $k[U_{i}] = k[Y_{1}, \dots, Y_{m}]$, with
\[
Y_{1} = \frac{X_{0}}{X_{i}^{a_{0}}}, \quad \dots, \quad Y_{m} 
      = \frac{X_{m}}{X_{i}^{a_{m}}}
\]
and $U_{i}$ is isomorphic to $\Aff^{m}$. The morphism 
\[
\begin{tikzcd}
\varphi \colon \Aff_{\{i\}}^{m} \arrow[r,"\sim"] & U_{i},
\end{tikzcd}
\]
is an isomorphism, with an inverse
\[
\begin{tikzcd}
\psi \colon U_{i} \arrow[r,"\sim"] & \Aff_{\{i\}}^{m}
\end{tikzcd}
\]
given by
\[
\psi(x_{0} : \ldots : x_{i} : \ldots : x_{m}) =
    (\frac{x_{0}}{x_{i}^{a_{0}}}, \dots, 1, \dots, \frac{x_{m}}{x_{i}^{a_{m}}}).
\]
Since $U_{i}$ is isomorphic to $\Aff^{m}$, the space $\PP(\w)$ is a 
compactification of the affine space $\Aff^{m}$.

\subsubsection{Action of $\GG_{m}$}
The action
\[
\begin{tikzcd}
\sigma \colon  \GG_{m} \times \Aff_{\{i\}}^{m} \arrow[r] & V_{i}
\end{tikzcd}
\]
is given by 
\[
\sigma(t).(x_{0}, \dots, 1, \dots x_{m}) = (t^{a_0}x_{0}, \dots,  t^{a_i}, \dots,  t^{a_m} x_{m}).
\]
Let $x = (x_{0}, \dots, x_{m})$ and similarly for $x'$. 
If $\sigma(t').x' = \sigma(t).x$, then $(t')^{a_i} = t^{a_i}$ and
$t' = \zeta^{-1} t$ with $\zeta \in \bmu_{a_{i}}$. We thus have
$(x_{0}', \dots, x_{m}') = (\zeta^{a_0} x_{0}, \dots, \zeta^{a_m} x_{m})$ and
\[
\sigma(t').x' = \sigma(t).x \quad \Longleftrightarrow \quad t'
              = \zeta^{-1} t \ \text{and} \ x' = \zeta.x, \quad \zeta \in \bmu_{a_{i}}. 
\]
Hence, the action $\sigma$ factors through
\[
(\GG_{m} \times \Aff_{\{i\}}^{m}) / \bmu_{a_{i}}
\]
with the action $\zeta.(t,x) = (\zeta^{-1} t, \zeta.x)$. The canonical homomorphism
\[
\begin{tikzcd}
\sigma^{\flat}  \colon  k[V_{i}] \arrow[r] & k[\Aff_{\{i\}}^{m}][T,T^{-1}]
\end{tikzcd}
\]
is equal, for $f$ $\w$-homogeneous, to
\[
\sigma^{\flat} \left( \frac{f}{X_{i}^{n}} \right) =
     f(X_{0}, \dots, 1, \dots, X_{m}) \cdot T^{\deg f - n a_{i}}
\]
which is injective, with image equal to $k[\Aff_{\{i\}}^{m}][T,T^{-1}]^{\bmu_{a_{i}}}$. Then:

\begin{proposition}
The action $\sigma$ induces isomorphisms
\[
\begin{tikzcd}
(\GG_{m} \times \Aff_{\{i\}}^{m})/\bmu_{a_{i}}  \arrow[r,"\thicksim"] & V_{i},
\end{tikzcd}
\quad
\begin{tikzcd}
k[V_{i}] \arrow[r, "\thicksim"] & k[\Aff_{\{i\}}^{m}][T,T^{-1}]^{\bmu_{a_{i}}},
\end{tikzcd}
\]
and $\sigma$ is an \'etale morphism.
\end{proposition}

\begin{proof}
See \cite[Th.~2.6.c, p.~12]{B-M}.\qed
\end{proof}

\noindent
{\bf Warning. } 
These isomorphisms are not surjective on the sets of rational points: think of 
the covering $\Aff^{1} \to \Aff^{1}$ given by $z \mapsto z^{2}$ !

\subsection{Rationality}
Let $k$ be a field. A point $y \in \PP(\w)$ is rational if and only if
$p^{- 1}(y)$ is invariant under the Galois group $\Gamma = \Gal(\bar{k}/k)$. 
We denote as usual the subset of rational points of $\PP(\w)$ by $\PP(\w)(k)$.
The orbit of $x = (x_{0}, \dots, x_{m}) \in \VV(\overline{k})$ with image
$p(x) = y$ is the rational curve
\[
C(x) =  p^{- 1}(y) = \sigma(\overline{k}^{\times}).x 
     =  \set{(\lambda^{a_0}x_{0}, \dots, \lambda^{a_m} x_{m})}
            {\lambda \in \overline{k}^{\times}}
        \subset \VV(\overline{k}).
\]

\begin{lemma}
\label{rationality}
Let $k$ be any field.
\begin{enumerate}
\item \label{RAT1}
      Let $x \in \VV$. Then
      \[
      p(x) \in \PP(\w)(k) \quad \Longleftrightarrow \quad C(x) \cap \VV(k) \neq \emptyset.
      \]
      In other words, the map $\xymatrix{p \colon \VV(k) \ar[r] & \PP(\w)(k)}$
      is surjective.
\item \label{RAT2}
      The map $p$ induces a bijection 
      $\xymatrix{\VV(k)/R \ar[r]^{\sim} & \PP(\w)(k)}$  where $R$ is the 
      equivalence relation whose classes are the subsets $C(x) \cap \VV(k)$.
\end{enumerate}
\end{lemma}

\begin{proof}
It is sufficient to prove the first assertion. See \cite[Lem.~6]{Perret} and 
\cite[Lemma~1.2]{Goto}.\qed
\end{proof}

\begin{lemma}
\label{NbPts}
Assume $k = \Fq$. Recall that $p$ is prime to all $a_{i}$.
\begin{enumerate}
\item \label{RAT3}
      If $x \in \VV(k)$, then $\card{C(x) \cap \VV(\Fq)} = \sigma(k^{\times}).x$ and
      \[
      \card{C(x) \cap \VV(\Fq)} = q - 1.
      \]
\item \label{RAT4}
      The map $p$ induces a bijection
      \[
      \xymatrix{\VV(k)/\sigma(k^{\times}) \ar[r]^{\sim} & \PP(\w)(k)}
      \]
\item \label{RAT5}
      We have
      \[
      \card{\PP(\w)(\Fq)} = \pi_{m}, \quad \text{with} \quad
      \pi_{m} = \card{\PP^{m}(\Fq)} = \frac{q^{m + 1} - 1}{q - 1}.
      \]
\end{enumerate}
\end{lemma}

\begin{proof}
\eqref{RAT3}: See Goto \cite[Prop.~1.3]{Goto} and  Perret \cite[Lem.~7]{Perret}.
Then \eqref{RAT4} follows from \eqref{RAT3} and Lemma~\ref{rationality}\eqref{RAT2}, 
whereas \eqref{RAT5} follows from~\eqref{RAT4}.\qed
\end{proof}

\begin{corollary}
Let $X$ be a hypersurface in a weighted projective space of dimension $m$ 
over~$\Fq$. Write $\card{X(\Fq)}$ for the number of $\Fq$-rational points 
on $X$ and $\card{(\Cone X)(\Fq)}$ for the number of affine solutions for the 
defining equation of $X$ in $\Aff^{m + 1}$. Then
\[
\card{(\Cone X)(\Fq)} = (q - 1) \card{X(\Fq)} + 1.
\]
\end{corollary}

\begin{proof}
See \cite[Cor.~1.4]{Goto}.\qed
\end{proof}

If $X$ is a hypersurface of degree $d$ over $\Fq$ in $\PP^{m}$ with $m \geq 1$,
then Serre's inequality is
\[
\card{X(\Fq)} \leq d q^{m - 1} + \pi_{m - 2}
\]
(recall that $\pi_{m - 2} = 0$), and hence,
\[
\card{(\Cone X)(\Fq)} \leq d q^{m} - (d - 1) q^{m - 1}.
\]

The following result is a bit amazing:

\begin{corollary}
Let $\Aff_{\{i\}}^{m}$ the affine hypersurface of $\VV$ with equation 
$X_{i} = 1$, and
\[
\begin{tikzcd}
p \colon \Aff_{\{i\}}^{m} \arrow[r] & \Aff_{\{i\}}^{m}/\bmu_{a_{i}}
\end{tikzcd}
\]
the quotient map under the action of type 
\[
\frac{1}{a_{i}}(a_{0}, \dots, \widehat{a_{i}}, \dots, a_{m}).
\]
Let $Z_{i}$ be the scheme $\Aff_{\{i\}}^{m}/\bmu_{a_{i}}$. Then
\[
\card{Z_{i}(\Fq)} = q^{m}.
\]
\end{corollary}

\begin{proof}
This is a consequence of \eqref{motivic} and of Lemma~\ref{NbPts}\eqref{RAT5}.\qed
\end{proof}

To be less amazed, observe that if $q$ is odd and $Z = \Aff^{1}/\bmu_{2}$,
then \hbox{$\card{Z(\Fq)} = q$.}

\subsection{Weighted forms}

\subsubsection{Definition}

Since the natural homomorphism $\pi^{*}$ defines an isomorphism
\[
\begin{CD}
\reg_{\PP(\w)}(U) & @>{\sim}>> & \pi_{*}(\reg_{\PP^{m}})^{G}(U) = \reg_{\PP^{m}}(\pi^{-1}(U))^{G},
\end{CD}
\]
for any open set $U \subset \PP(\w)$, we have a homomorphism of graded rings
\[
\begin{CD}
\pi^{\flat} \colon k[X_{0}, \dots, X_{m}] & @>>> & k[X_{0}^{a_0}, \dots, X_{m}^{a_m}]
\end{CD}
\]
such that $\pi^{\flat}(X_{i}) = X_{i}^{a_i}$. This leads to the isomorphism
\[
\begin{CD}
S(\w) & @>{\sim}>> & k[X_{0}^{a_0}, \dots, X_{m}^{a_m}] = k[X_{0}, \dots, X_{m}]^{G},
\end{CD}
\]
see \cite[p.~37]{Dgv1982} and \cite[Lemma~4.2.1]{Hosgood}. 

Henceforth we write $X = (X_{0}, \dots, X_{m})$ and denote by $k[X]$ the 
algebra of polynomials in $(X_{0}, \dots, X_{m})$. A polynomial $f \in k[X]$ 
is \emph{quasi-homogeneous} (or \emph{weighted homogeneous}, or a 
\emph{weighted form}) of $\w$-degree $d$ (or of degree $d$ w.r.t.~$\w$) if
\[
f(\lambda^{a_0}X_{0}, \dots, \lambda^{a_m} X_{m}) = \lambda^{d}(X_{0}, \dots, X_{m}).
\] 
We denote by $k[X]_{d}$ the vector space of homogeneous polynomials of 
degree~$d$, and by $k^{\w}[X]_{d}$ the vector space of quasi-homogeneous 
polynomials of $\w$-degree $d$. Now
\[
f \in k^{\w}[X]_{d} \quad \Longrightarrow \quad \pi^{*}f \in k[X]_{d}. 
\]
For a monomial $m = X_{0}^{r_0} \dots X_{m}^{r_m}$, we have
\[
m(\lambda^{a_0}X_{0}, \dots, \lambda^{a_m} X_{m}) =
     \lambda^{a_0 r_0}X_{0}^{r_0} \dots \lambda^{a_m r_m} X_{m}^{r_m}
\]
hence, $m \in k^{\w}[X]_{d}$ with
\[
a_{0} r_{0} + \dots + a_{m} r_{m} = d,
\]
and the dimension of $k^{\w}[X_{0}, \dots, X_{m}]_{d}$ is equal to
\[
\set{(r_{0}, \dots, r_{m}) \in \NN^{m}}{a_0 r_0 + \dots + a_m r_m = d}.
\]
This number can be calculated with the help of Ehrhart polynomials (see~\cite{Be}).

Every $f \in k^{\w}[X]_{d}$ defines a hypersurface
\[
Y = Y_{f} = \set{(x_{0} : \ldots : x_{m}) \in \PP(\w)}{f(x_{0}, \dots, x_{m}) = 0},
\]
and we associate also to $f$ the projective hypersurface of degree $d$:
\[
X = X_{f} = \set{(x_{0} : \ldots : x_{m}) \in \PP^{m}}{\pi^{*}f(x_{0}, \dots, x_{m}) = 0},
\]
and the morphism $\pi \colon  \PP^{m} \longrightarrow \PP(\w)$ induces a morphism
\[
\begin{CD}
\pi \colon X_{f} & @>>> & Y_{f}
\end{CD}
\]
providing a diagram
\[
\xymatrix{X_{f} \ar[rr]^{\pi} \ar[rd]^{p} & & Y_{f} \\ & X_{f}/G \ar[ru]^{\sim} } 
\]
and the morphism $\pi$ enjoys the properties of Proposition~\ref{GGQ}.

\subsubsection{Weighted binary forms}

Let $\w = (a_{0}, a_{1})$ and assume $a_{1} > 1$. We work with the weighted 
projective line $\PP(a_{0}, a_{1})$. It is known that
$\PP(a_{0}, a_{1}) \simeq \PP^{1}$, see \cite[p. 38]{Dgv1982}. If $P_{0} = (0:1)$
then $\PP(a_{0}, a_{1}) = D_{0} \cup \{P_{0}\}$.

\begin{proposition}[D'Alembert's theorem for weighted binary forms]
Let $\w = (1, a_{1})$. Let $f \in k[X_{0}, X_{1}]$ be a binary weighted form 
with weighted degree $d$, where $a_{1} \mid d$. Then the finite set
\[
X_{f} = \set{(x_{0}, x_{1}) \in \PP(1, a_{1})}{f(x_{0}, x_{1}) = 0}
\]
satisfies
\[
\card{X_{f}} \leq \frac{d}{a_{1}}.
\]
\end{proposition}

\begin{proof}
Let $\w = (a_{0}, a_{1})$, and assume $a_{0} a_{1} \mid d$. We have
\begin{align*}
f(x_{0}, x_{1}) & = \sum_{r_{0},r_{1}} c_{r_{0},r_{1}} x_{0}^{r_{0}} x_{1}^{r_{1}} 
                     \quad (a_{0} r_{0} +  a_{1} r_{1} = d) \\
\intertext{and in decreasing powers of $x_{1}$:}
f(x_{0}, x_{1}) & =  c_{0,d/a_{1}} x_{1}^{d/a_{1}} + \dots 
                     + c_{r_{0},r_{1}} x_{0}^{r_{0}} x_{1}^{r_{1}} + \dots 
                     + c_{d/a_{0},0} x_{0}^{d/a_{0}}.
\end{align*}
Notice that $a_{0}$ divides every index $r_{1}$. If $x_{0} = 0$ the equation 
reduces to
\[
c_{0,d/a_{1}} x_{1}^{d/a_{1}} = 0
\]
and the equation has exactly one solution if $c_{0,d/a_{1}} = 0$, namely 
$P_{0}$, and none otherwise. In $D_{0}$, we have as well
\begin{align*}
\frac{f(x_{0}, x_{1})}{x_{0}^{d/a_{0}}} 
         &= c_{0,d/a_{1}} \frac{x_{1}^{d/a_{1}}}{x_{0}^{d/a_{0}}} + \dots
            + c_{r_{0},r_{1}} \frac{x_{1}^{r_{1}}}{x_{0}^{a_{1} r_{1}/a_{0}}} + \dots 
            + c_{d/a_{0},0} \\
         &= c_{0,d/a_{1}} \left( \frac{x_{1}^{a_{0}}}{x_{0}^{a_{1}}} \right)^{d/a_{0} a_{1}} + \dots
            + c_{r_{0},r_{1}} \left( \frac{x_{1}^{a_{0}}}{x_{0}^{a_{1}}} \right)^{r_{1}/a_{0}} + \dots 
            + c_{d/a_{0},0}  \\
         &= f_{0}(u), \\
\intertext{with $u = x_{1}^{a_{0}}/x_{0}^{a_{1}}$, and}
f_{0}(u) &= c_{0,d/a_{1}} u^{d/a_{0} a_{1}} + \dots 
            + c_{r_{0},r_{1}} u^{r_{1}/a_{0}} + \dots + c_{d/a_{0},0}.
\end{align*}
This is a polynomial of degree $\leq d/a_{0} a_{1}$ with strict inequality if
$c_{0,d/a_{1}} = 0$. If $\w = (1, a_{1})$, the morphism 
$\varphi \colon U_{0} \to \Aff^{1}$ given by
\[
\varphi(x_{0} : x_{1}) = u = \frac{x_{1}}{x_{0}^{a_{1}}}
\]
is an isomorphism, with inverse morphism given by $u \mapsto (1 : u)$, and 
$\card{X_{f}} \leq d/a_{1}$. \qed
\end{proof}

\subsubsection{Weighted ternary forms}
\label{app:ore}

We are interested on weighted projective plane curves in the weighted 
projective plane $\PP(1,a_{1}, a_{2})$, that is, $m = 2$ and 
$\w = (1, a_{1}, a_{2})$.  We assume $1 < a_{1} < a_{2}$. Recall the notation:
the morphism 
\[
\xymatrix{\psi \colon U_{0} \ar[r] & \Aff^{2}_{0}}
\]
\[
\psi(x_{0} : x_{1} : x_{2}) = (1,y_{1}, y_{2}),
\]
where
\[
y_{1} = \frac{x_{1}}{x_{0}^{a_{1}}}, \quad 
y_{2} = \frac{x_{2}}{x_{0}^{a_{2}}} \, ,
\]
corresponds to the morphism of algebras
\[
\begin{CD}
\psi^{\flat} \colon k[X_{0},X_{1},X_{2}] & @>>> & k[Y_{1}, Y_{2}]
\end{CD}
\]
where
\[
Y_{1} = \frac{X_{1}}{X_{0}^{a_{1}}}, \quad
Y_{2} = \frac{X_{2}}{X_{0}^{a_{2}}} \, .
\]
The morphism $\psi$ is an isomorphism, with inverse 
$\varphi \colon \Aff^{2}_{0} \to U_{0}$ given by
\[
\psi(1,y_{1}, y_{2}) = (1 : y_{1} : y_{2}).
\]

The complement of $U_{0}$ is the weighted projective line $\PP(a_{1}, a_{2})$,
and $\PP(1,a_{1}, a_{2})$ is a compactification of the affine plane.

Recall that Ore's inequality (1922) for forms is the following: Let $f$ be a 
form in $m + 1$ variables, of degree $d$, defined over $\Fq$.  Define
\[
X_{f} = \set{x \in \PP^{m}}{f(x) = 0},
\]
and $(X_{f})^{\aff} = X_{f} \cap U_{0}$. Then
\[
\card{(X_{f})^{\aff}(\Fq)} \leq d q^{n - 1}.
\]

\begin{proposition}
Let $\w = (1, a_{1}, a_{2})$ and $f \in k[X_{0}, X_{1}, X_{2}]$ a ternary 
weighted form with weighted degree $d$, where $a_{1} a_{2} \mid d$. Define
\[
X_{f} = \set{(x_{0}, x_{1}, x_{2}) \in \PP(1,a_{1}, a_{2})}{f(x_{0}, x_{1}, x_{2}) = 0}.
\]
\begin{enumerate}
\item \label{WHO1}
      \textup{(}Ore's inequality for weighted ternary forms\textup{)}.
      Let $(X_{f})^{\aff} = X_{f} \cap U_{0}$. Then
      \[
      \card{(X_{f})^{\aff}(\Fq)} \leq \frac{d}{a_{1}} q.
      \]
\item \label{WHO2}
      We have
      \[
      \card{X_{f}(\Fq)} \leq \frac{d}{a_{1}} q + 1.
      \]
\end{enumerate}
\end{proposition}

\begin{proof}
Proof of \eqref{WHO1}: we write
\[
f(X_{0}, X_{1}, X_{2})  = \sum c_{i} X_{0}^{p_{i}} X_{1}^{q_{i}} X_{2}^{r_{i}},
\quad
p_{i} + a_{1} q_{i} + a_{2} r_{i} = d.
\]
The general term of $f/X_{0}^{d}$ is
\[
\frac{X_{0}^{p_{i}} X_{1}^{q_{i}} X_{2}^{r_{i}}} {X_{0}^{p_{i} + a_{1} q_{i} + a_{2} r_{i}}}
= \frac{X_{1}^{q_{i}}}{X_{0}^{a_{1} q_{i}}} \cdot \frac{X_{2}^{r_{i}}}{X_{0}^{a_{2} r_{i}}}
= Y_{1}^{q_{i}} Y_{2}^{r_{i}}.
\]
If $p_{1} = r_{1} = 0$, then $q_{1} = d/a_{2}$, if $p_{2} = q_{2} = 0$, then 
$q_{2} = d/a_{2}$, and if $q_{0} = r_{0} = 0$, then $p_{0} = d$. Hence,
\begin{align*}
f(X_{0}, X_{1}, X_{2})                   &= c_{1} X_{1}^{d/a_{1}} + c_{2} X_{2}^{d/a_{2}} + \dots + c_{0} X_{0}^{d}, \\
\intertext{and}
\frac{f(X_{0}, X_{1}, X_{2})}{X_{0}^{d}} &= c_{1} Y_{1}^{d/a_{1}} + c_{2} Y_{2}^{d/a_{2}} + \dots + Y_{1}^{q_{i}} Y_{2}^{r_{i}} + \dots + c_{0}.
\end{align*}
This is a bivariate polynomial of degree $\leq d / a_{1}$ in $\Aff^{2}$.  We get
the result using the usual Ore inequality. For a proof of~\eqref{WHO2}, see 
Theorem~\ref{serreforplanes} in the main text.\qed
\end{proof}

%\bibliographystyle{hplaindoi}
%\bibliography{appx}

\providecommand{\bysame}{\leavevmode\hbox to3em{\hrulefill}\thinspace}
\providecommand{\xxMR}[2]{\relax\ifhmode\unskip\space\fi
  \href{http://www.ams.org/mathscinet-getitem?mr=#2}{MR~#1}}
\providecommand{\xxZBL}[1]{\relax\ifhmode\unskip\space\fi
  \href{http://www.emis.de/cgi-bin/MATH-item?#1}{ZBL~#1}}
\providecommand{\xxJFM}[1]{\relax\ifhmode\unskip\space\fi
  \href{http://www.emis.de/cgi-bin/JFM-item?#1}{JFM~#1}}
\providecommand{\xxARXIV}[2]{\relax\ifhmode\unskip\space\fi
  \href{http://arxiv.org/abs/#2}{arXiv:#1}}
\providecommand\bibmarginpar{\leavevmode\marginpar}
\providecommand{\href}[2]{#2}

\end{subappendices}
\end{document}